\documentclass[article,12pt]{amsart}

\usepackage{amsfonts}
\usepackage{amsmath}
\usepackage{graphicx}
\usepackage{color}
\usepackage{epsfig}
 \usepackage{dsfont}
\usepackage{float}

\newcommand{\veps}{\varepsilon}

\newcommand{\ind}{\mbox{1}\kern-.25em \mbox{I}}

\def\build#1_#2^#3{\mathrel{\mathop{\kern 0pt#1}\limits_{#2}^{#3}}}

\def\videbox{\mathbin{\vbox{\hrule\hbox{\vrule height1ex \kern.5em
\vrule height1ex}\hrule}}}

\numberwithin{equation}{section}
\theoremstyle{plain}
\newtheorem{thm}{Theorem}[section]
\newtheorem{rem}{Remark}[section]

\newtheorem{lem}{Lemma}[section]

\newtheorem{prop}{Proposition}[section]
\newtheorem{ex}{Example}[section]

\def\Zbar{\overline{Z}}

\def\Esp{\mathbb{E}}
\def\Var{\mathbb{V}}

\def\R{\mathbb{R}}

\def\norm#1{\left\| #1 \right\|}
\def\abs#1{\left| #1 \right|}
\def\acc#1{\left\{ #1 \right\}}
\def\pa#1{\left( #1 \right)}
\def\pab#1{\bigl( #1 \bigr)}

\def\cro#1{\left[ #1 \right]}

\def\ind{\mathbb{1}}

\def\norm#1{\left\| #1 \right\|}
\def\abs#1{\left| #1 \right|}
\def\acc#1{\left\{ #1 \right\}}
\def\pa#1{\left( #1 \right)}
\def\pab#1{\bigl( #1 \bigr)}

\def\cro#1{\left[ #1 \right]}

\def\build#1_#2^#3{\mathrel{
\mathop{\kern 0pt#1}\limits_{#2}^{#3}}}
\def\tend_#1^#2{\mathrel{
\mathop{\kern 0pt\longrightarrow}\limits_{#1}^{#2}}}
\def\ntinf{n \rightarrow \infty}

\def\tendloi{\tend_{\ntinf}^{\mathcal{L}}}

\def\bkR{\mathbb{R}}
\def\bkE{\mathbb{E}}
\def\indic{{\large\mathds{1}}}

\def\thetachap{\widehat\theta}

\ifx\undefined\ly 
\gdef\beginguillemets{\leavevmode\raise0.3ex%
         \hbox{{$\scriptscriptstyle\langle\!\langle\,$}\nobreak\ignorespaces}} 
\else 
\gdef\beginguillemets{\leavevmode\hbox{\ly(\kern-0.20em(\kern+0.20em}\nobreak} 
\fi 
\ifx\undefined\ly 
\gdef\endguillemets{\ifdim\lastskip>\z@\unskip\penalty\@M\fi 
         \leavevmode\raise0.3ex%
         \hbox{{$\scriptscriptstyle\,\rangle\!\rangle$}}} 
\else 
\gdef\endguillemets{\nobreak\leavevmode\hbox{\kern+0.20em\ly)\kern-0.20em)}} 
\fi

\email{Antoine.Godichon@u-bourgogne.fr}
\email{Bruno.Portier@insa-rouen.fr}

\begin{document}
\title[Estimating the parameters of a truncated spherical distribution]
{An averaged projected Robbins-Monro algorithm for estimating
the parameters of a truncated spherical distribution}
\author{Antoine Godichon-Baggioni}
\address{Universit\'e de Bourgogne, Institut de Math\'ematiques de Bourgogne, 9 Avenue Alain Savary, 21000 Dijon, France}
\author{Bruno Portier}
\address{Normandie Universit\'e, INSA de Rouen, Laboratoire de Math\'ematiques,  
BP 08 - Avenue de l'Universit\'e, 76131 Saint-Etienne-du-Rouvray, France}

\maketitle

\begin{abstract}

The objective of this work is to propose a new algorithm to fit a sphere on a noisy 3D point cloud distributed around a complete or a truncated sphere.
More precisely, we introduce a projected Robbins-Monro
algorithm and its averaged version for estimating the center and the radius of the sphere.
We give asymptotic results such as the almost sure convergence of these algorithms as well as the asymptotic normality of the averaged algorithm.
Furthermore, some non-asymptotic results will be given, such as the rates of convergence in quadratic mean.
Some numerical experiments show the efficiency
of the proposed algorithm on simulated data for small to moderate sample sizes.
\end{abstract}

\medskip

\noindent \textbf{Keywords.}  Projected Robbins-Monro algorithm, Sphere fitting, Averaging, Asymptotic properties.

\section{Introduction}\label{SectionI}
\vspace{2ex}

Primitive shape extraction from data is a recurrent problem in many research fields such as archeology \cite{Thom}, medicine \cite{Zhang07}, mobile robotics \cite{Martins08}, motion capture \cite{Shafiq01} and computer vision \cite{Rabbani06, Liu14}. 
This process is of primary importance since it provides a high level information on the data structure.  
\vspace{1ex}

First works focused on the case of 2D shapes (lines, circles), but recent technologies enable to work with three dimensional data. 
For instance, in computer vision, depth sensors provide 3D point clouds representing the scene in addition to usual color images. 
In this work, we are interested in the estimation of the center $\mu \in \mathbb{R}^3$ and the radius $r > 0$ of a sphere from a set of 3D noisy data.
In practical applications, only a discrete set of noisy measurements is available. Moreover, sample points are usually located only near a portion of the spherical surface. Two kinds of problem can be distinguished: shape detection and shape fitting.
\vspace{1ex}

Shape detection consists in finding a given shape in the whole data without any prior knowledge on which observations belong to it. In that case, the data set may represent several objects of different nature and may therefore contain a high number of outliers.
Two main methods are used in practise to solve this problem. 
The Hough transform \cite{Abuzaina2013} performs a discretization of the parameter space. 
Each observation is associated to a set of parameters corresponding to all possible shapes that could explain the sample point. 
Then, a voting strategy is applied to select the parameter vectors of the detected shapes. 
The advantage of this method is that several instances of the shape can be detected. 
However, a large amount of memory is required to discretize the parameter space, especially in the case of three dimensional models. 
The RANSAC (RANdom SAmple Consensus) paradigm \cite{Ransac_Rischler1981,Ransac_Schnabel2007} is a probabilistic method based on random sampling. 
Observations are randomly selected among the whole data set and candidate models are generated. 
Then, shapes can be detected thanks to an individual scoring scheme. 
The success of the method depends on a given probability related to the number of sampling and the fraction of points belonging to the shape. 
\vspace{1ex}
 
The shape fitting problem assumes that all the data points belong to the shape. For example, spherical fitting techniques have been used in several domains such as industrial inspection \cite{Jiang98}, GPS localization \cite{Beck2012}, robotics \cite{von2005constructive} and 3D modelling \cite{tran2015extraction}. 
Geometric and algebric methods have been proposed \cite{landau1987estimation, Rusu03, AlSharadqah14} for parameters estimation. 
Moreover, let us note that fitting methods are generally applied for shape detection as a post-processing step in order to refine the parameters of the detected shapes \cite{tran2015extraction}. 
\vspace{1ex}

In a recent paper, Brazey and Portier \cite{Brazey2014} introduced a new spherical probability density function 
belonging to the family of elliptical distributions, and designed to model points spread near a spherical surface. 
This probability density function depends on three parameters, namely a center $\mu \in \mathbb{R}^3$, a radius $r > 0$ and a dispersion parameter $\sigma > 0$.  
In their paper, the model is formulated in a general form in $\mathbb{R}^d$. 
To estimate $\mu$ and $r$, a backfitting algorithm (see e.g. \cite{breiman1985estimating}) similar to the one used in \cite{landau1987estimation} is employed. 
A convergence result is given in the case of the complete sphere.
However, no result is established in the case of a truncated sphere while simulations showed the efficiency of the algorithm.
\vspace{1ex}
 
The objective of this work is to propose a new algorithm to fit a sphere on a noisy 3D point cloud distributed around a complete or a truncated sphere.
We shall assume that the observations are independent realizations of a random vector $X$ defined as 
\begin{equation}\label{Model_intro}
X = \mu + r \, W \,U_\Omega , 
\end{equation}
where $W$ is a positive real random variable
such that $\mathbb{E}\cro{W} = 1$,
$U_\Omega$ is uniformly distributed on a measurable subset $\Omega$ of the unit sphere of $\mathbb{R}^3$,
$W$ and $U_\Omega$ are independent.
Parameters $\mu \in \mathbb{R}^3$ and $r>0$ are respectively the center and the radius of the sphere we are trying to adjust to the point cloud.
Random variable $W$ allows to model the fluctuations of points 
in the normal direction of the sphere.
When $\Omega$ coincides with the complete sphere, then the distribution of $X$ is spherical (see e.g. \cite{muirhead2009aspects}).
Indeed, if we set $Y = (X - \mu)/ r$, then the distribution of $Y$ is rotationally invariant. 
\vspace{1ex}

We are interested in estimating center $\mu$ and radius $r$.
As $\norm{U_\Omega} = 1$, we easily deduce from
(\ref{Model_intro}) that
\begin{eqnarray}
\mu & = & \mathbb{E}\left[X - r \, \dfrac{(X - \mu)}{\norm{X - \mu}}\right] 
\label{Equation_mu}\\
r  & = & \mathbb{E}\left[\norm{X - \mu}\right].
\label{Equation_r}
\end{eqnarray}
It is clear that from these two equations, we cannot deduce
explicit estimators of parameters $\mu$
and $r$ using the method of moments 
since each parameter depends on the other.
To overcome this problem, we can use a backfitting type algorithm
(as in \cite{Brazey2014}) 
or introduce a recursive stochastic algorithm, which seems well-suited for 
this problem since equations 
(\ref{Equation_mu}) and (\ref{Equation_r}) can also be derived from the local minimization of the following quadratic criteria
\begin{eqnarray}
G(\mu,r)  := \dfrac{1}{2}\,\bkE\cro{(\norm{X-\mu} - r)^2}.
\end{eqnarray}

Stochastic algorithms, and more precisely Robbins-Monro algorithms, 
are effective and fast methods (see e.g. \cite{DUF1997, Kushner2003, robbins1951}). They do not need too much computational efforts and can easily be updated, which make of them good candidates to deal with big data for example. 
However, usual sufficient conditions to prove the convergence of this kind of algorithm are sometimes not satisfied and it is necessary to 
modify the basic algorithm.
We can, for example, introduce 
a projected version of the Robbins-Monro algorithm 
which consists in keeping the usual estimators in a nice subspace with the help of a projection.
Such an algorithm has been recently considered in \cite{bercu2012robbins} 
and \cite{lacoste2012simpler}.
\vspace{1ex}

In this paper, due to the non global convexity of function $G$, 
we estimate parameters $\mu$ and $r$ using a projected 
Robbins-Monro algorithm.
We also propose an averaged algorithm
which consists in averaging the projected algorithm.
In general, this averaged algorithm 
allows to improve the rate of convergence of the basic estimators, or to reduce the variance, or not to have to make a good choice of the step sequence, which can be as exhaustive as to estimate the parameters.
It is widely used when having to deal with Robbins-Monro algorithms 
(see \cite{polyak1992acceleration} or \cite{pelletier1998almost} amoung others).
\vspace{1ex}

This paper is organized as follows.
In Section 2, we specify the framework and assumptions.
After a short explanation on the non-convergence of the Robbins-Monro algorithm, the projected algorithm and its averaged version are introduced in Section~3.
Section 4 is concerned with the convergence results.
Some simulation experiments are provided in Section 5, showing 
the efficiency of the algorithms.
Proofs of the different results are postponed in Appendix.
\section{Framework and assumptions} \label{section_framework}

We consider in this paper a more general framework than the one described in the introduction.
Let $X$ be a random vector of $\bkR^d$ with $d\geq 2$.
Let $F$ denotes the distribution of $X$.
We assume that $X$ can be decomposed under the form
\begin{equation}\label{Def_X}
X = \mu + r\,W\,U_\Omega .
\end{equation}
where $\mu\in\bkR^d$, $r>0$, 
$W$ is a positive real continuous random variable (and with a bounded density if $d=2$),
$U_\Omega$ is uniformly distributed on a measurable subset
$\Omega$ of the unit sphere of $\bkR^d$. Moreover, let us suppose that $W$ and $U_\Omega$ 
are independent. 
\vspace{1ex}

Model (\ref{Def_X}) allows to model a point cloud of $\bkR^d$ 
spread around a complete or truncated sphere of center $\mu\in\bkR^d$
and radius $r>0$. 
Random vector $U_\Omega$ defines the position of the points
on the sphere and random variable $W$ defines the fluctuations
in the normal direction of the sphere.
As mentioned in the introduction, when $\Omega$ is the complete unit sphere, then the distribution of $X$ is spherical. 

\vspace{1ex}

When $W$ satisfies the condition $\bkE\cro{W} = 1$, 
 the radius $r$ is identifiable and can be directly estimated.
Indeed, since $\norm{U_\Omega} = 1$, then
$\norm{X-\mu} = rW$ and $\bkE\cro{\norm{X-\mu}}~ = ~ r\,\bkE\cro{W} ~= ~r$.
However, this condition is sometimes not satisfied (as in \cite{Brazey2014})
and only $r^\star := r\,\bkE\cro{W}$ can be estimated.
Therefore, in what follows, we are interested in estimating
$\theta : = \pa{\mu^T , r^{\star}}^T$,
which will be denoted by $(\mu , r^\star)$ for the sake of simplicity. 
\vspace{1ex}

We suppose from now that the following assumptions are fulfilled:
\begin{itemize}
\item \textbf{Assumption [A1].} 
The random vector $X$ is not concentrated around $\mu$:
\[\displaystyle\mathbb{E}\cro{\norm{X-\mu}^{-2}}  < \infty .\] \\
\item \textbf{Assumption [A2].} The random vector $X$ admits a second moment:
\[\displaystyle\mathbb{E}\cro{ \norm{X-\mu}^{2}} < \infty .\]
\end{itemize}
\vspace{1ex}

These assumptions ensure that the values of $X$ are concentrated around the sphere and not around the center $\mu$, without 
in addition too much dispersion. 
This framework totally corresponds to the real situation 
that we want to model.
Moreover, using (\ref{Def_X}), Assumptions [A1] and [A2] 
reduce to assumptions on $W$. 
More precisely, [A1] reduces to $\bkE\cro{W^{-2}} < \infty$ 
 and [A2] to $\bkE\cro{W^{2}} < \infty$.
\vspace{1ex}

Let us now introduce two examples of distribution allowing to model points 
spread around a complete sphere and satisfying assumptions [A1] and  [A2].

\begin{ex}\label{ExampleUn}
Let us consider a random vector $X$ taking values in $\bkR^d$ with a distribution 
absolutely continuous with respect to the Lebesgue measure,
with a probability density function $f_\delta$ defined for all $\delta >0$ by
\begin{equation} \label{Dens_Bounded}
f_{\delta}(x) \ =\ \dfrac{C_d}{\norm{x - \mu}^{d-1}}\, 
\indic_{\displaystyle\acc{\norm{x - \mu}/r \in [1 - \delta\,,\, 1 + \delta]}} , 
\end{equation}
where $C_d$ is the normalization constant.
Then, we can rewrite $X$ under the form (\ref{Def_X}) with
$U_\Omega = U$, $W \sim \mathcal{U}([1-\delta,1+\delta])$ 
and $\bkE\cro{W} = 1$ for any $\delta > 0$.
\end{ex}

\begin{ex}\label{ExampleDeux}
Let us consider the probability density function introduced in \cite{Brazey2014}.
It is defined for any $x\in\bkR^d$ by
\begin{eqnarray}\label{Dens_BP}
f(x)  = K_d\,\exp\pa{- \dfrac{1}{2\sigma^2} \pab{\norm{x - \mu} - r}^2}, 
\end{eqnarray}
where $K_d$ is the normalization constant.
Then, a random vector $X$ with probability density function $f$ can be rewritten under the 
form (\ref{Def_X}), with 
$\bkE\cro{W} \not=1$, but 
$\bkE\cro{W}$ is closed to $1$ when the variance $\sigma$ is negligible compared to the radius $r$.
\end{ex}
\def\Omegabar{\overline\Omega}

To obtain points distributed around a truncated sphere, 
it is sufficient to modify the previous densities
by considering densities of the form
$f_{\Omegabar}(x) = 
C_{\Omegabar} f(x)  
\indic_{\displaystyle\acc{(x - \mu) \in \Omegabar}}$ 
where $\Omegabar$ is the set of points of $\bkR^d$ whose polar 
coordinates are given by $(\rho, \theta_1, \ldots, \theta_{d-1} \in \bkR^*_+\times\Theta )$
where $\Theta$ defines the convex part $\Omega$ 
of the surface of the unit sphere of $\bkR^d$ we want to consider.

\section{The algorithms} \label{section_params_est}
\def\Fn{\mathcal{F}_{n}}

We present in this section two algorithms for estimating 
the unknown parameter $\theta$ which can be seen as a local minimizer (under conditions) of a function. Indeed, let us consider the function $G:~ \bkR^{d}\times ~ \bkR ~\longrightarrow ~\bkR$ 
defined for all $y = (z,a) \in \mathbb{R}^{d}\times \mathbb{R}_+^\ast$ 
by
\begin{eqnarray}\label{Def_G}
G(y) & := &  \dfrac{1}{2}\, \bkE\cro{(\norm{X- z} - a)^2}\ = \
 \dfrac{1}{2}\, \Esp\cro{g\pa{X,y}}, 
\end{eqnarray}
where we denote by  $g$ the function defined
for any $x\in \mathbb{R}^{d}$ and 
$y=(z,a)\in\mathbb{R}^{d} \times \mathbb{R}^*_+$ 
by $g(x,y)~:=~ \pa{\norm{x - z} - a}^2$.
The function $G$ is Frechet-differentiable and we denote by $\Phi$ its gradient, which is defined for all $y~=~(z,a)~\in~\mathbb{R}^{d} \times\mathbb{R}_+^\ast$ by
\begin{eqnarray}\label{Def_Gradient_G}
\Phi (y):= \nabla G(y) = \mathbb{E}\left[ \nabla_{y}g(X,y)\right] = \begin{pmatrix} z- \mathbb{E}\left[ X \right] -a\,\mathbb{E}
\cro{\dfrac{z-X}{\| z-X \|}} \\
a- \mathbb{E}\cro{\| z-X \|}.
\end{pmatrix}
\end{eqnarray}
From  (\ref{Def_X}) and definition of $\theta = \left( \mu , r^{\star}\right)$, we easily verify that
$\nabla G(\theta)=0$.
Therefore, since $\theta$ is a local minimizer of function $G$ (under assumptions) or a zero of $\nabla G$, an idea could be to introduce a stochastic gradient algorithm
for estimating $\theta$.

\subsection{The Robbins-Monro algorithm.} 
\ \ 
\vspace{1ex}

Let $(X_n)_{n\geq 1}$ be a sequence of independent and identically 
distributed random vectors of $\bkR^d$ following the same law as $X$
and let $(\gamma_n)_{n\geq 1}$ be a decreasing sequence of positive real numbers satisfying the usual conditions
\begin{equation}\label{Cond_gamma}
\sum_{n \geq 1}\gamma_n = \infty\quad \mbox{and} \quad \sum_{n \geq 1}\gamma_n^2 < \infty . 
\end{equation} 
When the functional $G$ is convex or verifies nice properties, a usual way to estimate the unknown parameter 
$\theta$ is to use the following recursive algorithm
\begin{equation}\label{Alg_RobMonro}
\theta_{n+1} = \theta_{n} - \gamma_{n}\,\nabla_{y} g(X_{n+1}, \theta_{n}),
\end{equation} 
with $\theta_1$ chosen arbitrarily bounded. 
The term $\nabla_{y} g\left( X_{n+1},\theta_{n}\right)$ can be seen as an estimate of the gradient of $G$ 
at $\theta_{n}$, and the step sequence $(\gamma_{n})$ controls the convergence of the algorithm. 
\vspace{1ex}

The convergence of such an algorithm is often established using the 
Robbins-Siegmund's theorem (see e.g. \cite{DUF1997})
and a sufficient condition to get it, is to verify that for any 
$y \in \mathbb{R}^{d}\times \mathbb{R}_+^\ast$, 
$\left\langle \Phi (y) , y- \theta \right\rangle > 0$
where $\left\langle . , .\right\rangle$ denotes the usual inner product and $\| . \|$ the associated norm.
However, we can show that this condition is only satisfied  
for $y$ belonging to a subset of $\mathbb{R}^{d}\times \mathbb{R}_+^\ast$
to be specified.
Thus, if at time $(n+1)$, the update of $\theta_n$ 
(using (\ref{Alg_RobMonro}))
leaves this subset, then 
it does not necessarily converge.
Therefore, we have to introduce a projected Robbins-Monro algorithm.

\subsection{The Projected Robbins-Monro algorithm} 
\ \ 
\vspace{1ex}

\def\calK{\mathcal{K}}
\def\calB{\mathcal{B}}
\def\mutild{\widetilde\mu}
\def\rtild{\widetilde r}
\def\vepstild{\widetilde\varepsilon}
\def\Rplus{\mathbb{R}^*_+}
\def\Rtild{\widetilde R}

Let $\calK$ be a compact and convex subset of $\bkR^d \times\Rplus$  
containing $\theta=(\mu,r^\ast)$ and
let $\pi: \mathbb{R}^{d}\times \Rplus \longrightarrow \calK$ be a projection 
satisfying
\begin{equation}\label{Cond_Proj}
 \left\{\begin{array}{l}
 \forall y,y' \in \mathbb{R}^{d} \times \Rplus,\ \ 
 \norm{\pi (y) - \pi (y')} \leq \left\| y - y' \right\| \\
 \forall y \notin \calK, \ \pi (y) \in \partial \calK \\
 \end{array}\right.
\end{equation}
where $\partial\calK$ is the frontier of $\calK$. An example will be given later.
\vspace{1ex}

\noindent
Then, we estimate $\theta$ using the following 
Projected Robbins-Monro algorithm (PRM), defined recursively by
\begin{equation}
\label{defprm} \thetachap_{n+1} = \pi \left( \thetachap_{n} - \gamma_{n} \nabla_{y}g\left( X_{n+1},\thetachap_{n} \right) \right) ,
\end{equation}
where $\thetachap_{1}$ is arbitrarily chosen in $\calK$, and $\pa{\gamma_{n}}$ is a decreasing sequence of positive real numbers satisfying (\ref{Cond_gamma}).
\vspace{1ex}

Of course the choice of subset $\calK$ and projector $\pi$ 
is crucial.
It is clear that if $\calK$ is poorly chosen for a given projector, 
the convergence of the projected algorithm towards $\theta$
will be slower, even if from a theoretical point of view,
we shall see in the next section dedicated to the theoretical results, that this algorithm is almost the same as the traditional Robbins-Monro algorithm since
the updates of $\thetachap_n$, ie. the quantities
$\pab{\thetachap_n - \gamma_{n}\,\nabla_{y}\, g
\pab{X_{n+1},\thetachap_{n}}}$, leave $\calK$ only a finite number of times. 
\vspace{1ex}

Let us now discuss the choice of $\calK$ and $\pi$.
The choice of $\calK$ is directly related to the following assumption
that we introduce to ensure
the existence of a compact subset on which the 
scalar product $\left\langle \Phi (y) , y- \theta \right\rangle$ is  positive.

\begin{itemize}
\item \textbf{Assumption [A3].} There are two positive constants $R_\mu$
and $R_r$ such that
for all $y = (z,a) \in \overline{\mathcal{B}(\mu, R_\mu)}\times
\overline{\mathcal{B}(r^\ast, R_r)}$, 
\begin{equation}\label{Cond_A3}
\sup_{z \in \overline{\mathcal{B}(\mu, R_\mu)}} 
\lambda_{\max}\pab{\Gamma(z)} < 
\frac{1 - \norm{\bkE\left[ U_\Omega\right]}^2 / A}{r^\ast + \frac{3}{2} R_r },
\end{equation}
with $A$ such that $\norm{\bkE\left[ U_\Omega\right]}^2 < A < 1$,
and $\lambda_{\max}(M)$ denotes the largest eigenvalue of matrix $M$, and
$$\Gamma(z) :=  \bkE\left[\frac{1}{\norm{X - z}}\pa{I_{d} - \frac{(X-z)(X-z)^{T}}{\norm{X-z}^{2}}}\right] .$$
\end{itemize}

\begin{rem}{\rm
The less the sphere is troncated, the more 
$\left\| \bkE\left[ U_{\Omega} \right]\right\|$ is close to $0$ and
the constraints on $R_\mu$ and $R_r$ are relaxed.
In particular, when the sphere is complete, ie. 
$U_\Omega = U$ where $U$ denotes the random vector
uniformly distributed on the whole unit sphere of $\bkR^d$, 
then $\bkE\left[U_\Omega\right] = 0$ and
Assumption [A3] reduces to 
\[
\sup_{z \in \overline{\mathcal{B}(\mu, R_\mu)}} 
\lambda_{\max}\pab{\Gamma(z)} < 
\frac{1}{r^\ast + \frac{3}{2} R_r }.
\]
}\end{rem}
\vspace{1ex}

The main consequence of Assumption [A3] is the following proposition
which is one of the key point to establish the convergence of the PRM algorithm.

\begin{prop}\label{propfortconv}
Assume that [A1] to [A3] hold. 
Then, there is a positive constant $c$ such that for all 
$y \in \overline{\calB(\mu, R_\mu)}\times
\overline{\calB(r^\ast, R_r)}$,
\[
\left\langle \Phi (y) , y - \theta \right\rangle \geq c \left\| y - \theta \right\|^{2} .
\]
\end{prop}

\begin{proof}
The proof is given in Appendix \ref{appendixA}.
\end{proof}

Assumption [A3] is therefore crucial but only technical.
It reflects the fact that the sphere is not too much
truncated and that the points are not too far away 
from the sphere which corresponds to the real situations
we want to model.

\noindent
In a general framework, this technical 
assumption is difficult to verify since it requires to specify the distribution of $X$.
In the case of distribution of Example \ref{ExampleUn} with $\delta < 1/10$, 
we can easily exhibit constant $R_\mu$  and $R_r$. 
Indeed taking $R_\mu = R_r =  r^\ast / 10$, 
then assumption [A3] holds.
When the distribution of $X$ is compactly supported
with a support included in $[1 - \delta, 1 + \delta]$,
it is fairly easy to find the constants provided that $\delta$ is small enough.
It is quite more difficult when dealing with distribution of Example 
\ref{ExampleDeux}. 
Nevertheless, topological results can ensure that these constants exist.
\vspace{1ex}

From constants $R_\mu$ and $R_r$ of Assumption [A3], it is then possible 
to simply define a projector $\pi$
which satisfies condition (\ref{Cond_Proj}).
Indeed, let us set $\calK = \calK_\mu \times\calK_r$ with
$\calK_\mu = \calB(\mu, R_\mu)$
and $\calK_r = \calB(r^*,R_r)$,
and define for any $y = (z,a)\in \bkR^d\times \bkR^*_+$ by $\pi (y):= \left( \pi_{\mu}(z) , \pi_{r}(a)\right)$, with
\begin{equation*}
\pi_{\mu}† (z) :=   \left\{\begin{array}{cl}
z & \text{if } z \in \calK_{\mu} \\
\mu + R_{\mu} \dfrac{(z - \mu)}{\norm{z-\mu}} & \text{otherwise }\\
\end{array}\right.
\end{equation*}
and
\begin{equation*}
\pi_{r} (a) :=   \left\{\begin{array}{cl}
a & \text{if } a \in \calK_{r_0} \\
r + R_{r} \dfrac{(a - r^{\star})}{\abs{a-r^{\star}}} & \text{otherwise }\\
\end{array}\right.
\end{equation*}
Such projector satisfies the requested conditions.
However, it is clear that this projector can not be implemented
since $\mu$ and $r^\ast$ are unknown.
We shall see in the simulation study how to overcome this problem.
\vspace{1ex}

We suppose from now that $\calK$ is a compact and convex 
subset of $\overline{\calB(\mu, R_\mu)} \times \overline{\calB(r^*,R_r)}$
such that $\theta \in \mathcal{K}$, but $\theta \notin \partial\mathcal{K}$, where $\partial \mathcal{K}$ is the frontier of $\mathcal{K}$, i.e there is a positive constant $d_{\min}$ such that $\overline{\mathcal{B}\left( \theta , d_{\min} \right)}\subset \mathcal{K}$.

\subsection{The averaged algorithm}
\ \ 
\vspace{1ex}

Averaging is a usual method to improve the rate of convergence of Robbins-Monro algorithms, or to reduce the variance, or finally not to have to make a good choice of the step sequence (see \cite{polyak1992acceleration}), but for the projected algorithms, 
this method is not widespread in the litterature.
In this paper, we improve the estimation of $\theta$ by adding
an averaging step to the PRM algorithm.  
Starting from the sequence $(\thetachap_n)_{n \geq 1}$ given by (\ref{defprm}), 
we introduce for any $n\geq 1$, 
\[
\overline{\theta}_{n} = \frac{1}{n}\sum_{k=1}^{n} \thetachap_{k} ,
\]
which can also be recursively defined by
\begin{equation}
\label{defmoy} \overline{\theta}_{n+1} = \overline{\theta}_{n} + 
\frac{1}{n+1}\pa{\thetachap_{n+1} - \overline{\theta}_{n}} ,\quad
\text{and} \quad \overline{\theta}_{1}=\widehat{\theta}_{1}.
\end{equation}
We shall see in the following two sections, the
gain provided by this algorithm. 

\section{Convergence properties} \label{section_properties}
\def\Var{\mathbb{V}\mbox{ar}}

We now give asymptotic properties of the algorithms. All the proofs are postponed in Appendix \ref{appendixB}. The following theorem gives the strong consistency of the PRM algorithm as well as properties on the number of times we really use the projection.
\begin{thm}\label{theops}
Let $\left( X_{n} \right)$ be a sequence of iid random vectors following the same law as $X$. 
Assume that [A1] to [A3] hold, then 
\[\lim_{n \rightarrow \infty} \| \thetachap_{n}  - \theta \| = 0 \quad a.s.\]
Moreover, the number of times the random vectors 
$\thetachap_n - \gamma_n\nabla_{y}g\left( X_{n+1},\widehat{\theta}_{n} \right)$ 
do not belong to $\calK$ is almost surely finite.
\end{thm}

The following theorem gives the rate of convergence in quadratic mean and 
the $L^{p}$~rates of convergence of the PRM algorithm (under conditions) as well as an upper bound of the probability that the random vector 
$\thetachap_n - \gamma_n\nabla_{y}g\left(  X_{n+1}, \thetachap_n \right)$ 
does not belong to $\calK$.
%
\begin{thm}\label{theol2l4}
Let $\left( X_{n} \right)$ be a sequence of iid random vectors following the same law as $X$. 
Assume that [A1] to [A3] hold and consider
a step sequence $\left( \gamma_{n} \right)$ of the form 
$\gamma_{n} = c_{\gamma}n^{-\alpha}$, with $c_{\gamma}> 0 $ 
and $\alpha \in ]1/2,1[$. 
 Then, there is a positive constant $C_1$ such that for all $n \geq 1$,
\[
\bkE\cro{\left\| \thetachap_{n} - \theta \right\|^{2}}  \leq 
\frac{C_1}{n^{\alpha}}.
\]
Moreover, for all positive integer $p$ such that $\bkE\cro{\norm{X-\mu}^{2p}} < \infty$, 
there is a positive constant $C_{p}$ such that for all $n \geq 1$,
\[
\bkE\left[ \left\| \thetachap_{n} - \theta \right\|^{2p} \right] \leq \frac{C_{p}}{n^{p\alpha}},
\]
and for all $n\geq 1$,
\[
\mathbb{P}\left[ \thetachap_n - \gamma_n\nabla_{y}g\left( X_{n+1},\thetachap_n \right) \notin \calK  \right] \leq \frac{C_{p}}{d_{\min}^{2p}\,n^{p\alpha}},
\]
where $d_{\min} := \inf_{y \in \partial K}\acc{\left\| y - \theta \right\|}$ and $\partial K$ is the frontier of $K$.
\end{thm}

We now focus on the asymptotic behavior of the averaged algorithm.
First of all, applying Theorem \ref{theops} and Toeplitz's lemma for example, we easily obtain
the strong consistency of the averaged estimator $\overline{\theta}_n$. 
Introducing the following assumption, we can specify 
its rate of convergence in quadratic mean as well as its asymptotic normality.
\vspace{1ex}

\noindent
\begin{itemize}
\item \textbf{Assumption [A4].}  The Hessian of $G$ at $\theta = (\mu, r^*)$,
denoted by $\Gamma_\theta$ and defined by
\[\Gamma_\theta := \begin{pmatrix}
I_d - r^*\bkE\cro{\dfrac{1}{\norm{X - \mu}} \pa{I_d - \dfrac{(X - \mu)\otimes(X - \mu)}{\norm{X - \mu}^2}}}
& \bkE\left[\dfrac{X - \mu}{\norm{X-\mu}}\right]\\
\bkE\left[\dfrac{X - \mu}{\norm{X-\mu}}\right]^T & 1 \\
\end{pmatrix}\]
 is a positive definite matrix.
\end{itemize}
\vspace{1ex}

Note that thanks to topological results, this assumption also implies Proposition \ref{propfortconv} but is not useful to obtain the constants $R_{\mu}$ and $R_{r}$. Nevertheless, this assumption is crucial to establish the results of the two 
following theorems but
it is satisfied as soon as the sphere is not too much truncated
and the dispersion around the sphere not too important
which corresponds to the real situations encountered.
Using model (\ref{Def_X}), $\Gamma_\theta$ rewrites under the form
\begin{equation}\label{eqhess}
\Gamma_{\theta}  = \begin{pmatrix}  
(I_d - \beta \left(I_d - \bkE\cro{U_\Omega\,U_\Omega^T}\right)  & \bkE\cro{U_\Omega} \\   
\bkE\cro{U_\Omega^T} &   1  \end{pmatrix}\quad\mbox{with}\quad 
\beta = \bkE\cro{W}\bkE\cro{W^{-1}}.
\end{equation}
When the sphere is complete, ie. $U_\Omega = U$, 
then $\bkE\cro{U_\Omega} = 0$, 
$\bkE\cro{U_\Omega\,U_\Omega^T} = (1/d) I_d$ and
$\lambda_{\min}(\Gamma_\theta) > 0$ as soon as $\beta < d/(d-1)$.
In the case of distribution of Example~\ref{ExampleUn},
we have $\beta = (\log(1+\delta) - \log(1-\delta))/(2\delta)$ and 
[A4] is satisfied as soon as $\delta$ is small enough.
In the case of distribution of Example \ref{ExampleDeux}, [A4] is satisfied
as soon as $r >\!\!> \sigma$.
When the sphere is not complete, we can easily show that a sufficient condition
to ensure [A4] is 
$\lambda_{\min}\pa{\mathbb{V}\mbox{ar}\cro{U_\Omega}} < 1 - 1/\beta$, where $\mathbb{V}\mbox{ar}\left[ U_{\Omega}\right]$ is the covariance matrix of the random variable $U_{\Omega}$.
In the case of the half sphere and $d=3$,  
we have $\lambda_{\min}\pa{\mathbb{V}\mbox{ar}\cro{U_\Omega}} = 1/12$
and $\Gamma_\theta$ is definite positive as soon as 
$\beta < 12 / 11$.
This condition holds for distribution of Example \ref{ExampleUn}
with $\delta < 0.4$ for instance, and distribution of Example \ref{ExampleDeux}
as soon as $r >\!\!> \sigma$.

\begin{thm}\label{theol2moy}
Let $(X_n)$ be a sequence of iid random vectors following the same law as $X$. 
Assume that [A1] to [A4] hold and consider
a step sequence $\left( \gamma_{n} \right)$ of the form 
$\gamma_{n} = c_{\gamma}n^{-\alpha}$, with $c_{\gamma}> 0 $ 
and $\alpha \in ]1/2,1[$. 
Moreover, suppose that $\bkE[\norm{X-\mu}^{12}] < \infty$.
Then there is a positive constant $C$ such that for all $n \geq 1$,
\[
\bkE\left[ \left\| \overline{\theta}_{n}  - \theta \right\|^{2} \right] \leq \frac{C}{n}.
\]
\end{thm}
 
With respect to results of Theorem \ref{theol2l4},
we clearly improve the rate of convergence in quadratic mean.
Note that the computed rate is the optimal one for such stochastic algorithms.
We finally give a central limit theorem
which can be useful to build confidence balls for the different parameters
of the sphere.
%
\begin{thm}\label{theotlc}
Let $(X_n)$ be a sequence of iid random vectors following the same law as $X$ 
and let us choose the step sequence $(\gamma_n)$ of the form
$\gamma_{n} = c_{\gamma}n^{-\alpha}$, with $c_{\gamma}> 0 $ 
and $\alpha \in ]1/2,1[$. 
 Assume that [A1] to [A4] hold and suppose that $\bkE[\norm{X-\mu}^{12}] < \infty$.
Then $\left( \overline{\theta}_{n} \right)$ satisfies
\begin{equation}\label{TLC}
\sqrt{n}\pa{\overline{\theta}_{n} - \theta} \tendloi
\mathcal{N}\pa{0, \Gamma_{\theta}^{-1}\Sigma\,\Gamma_{\theta}^{-1}}
\end{equation}
with 
\begin{equation}
\Sigma := \bkE\cro{
\begin{pmatrix}
\mu - X - r^* \dfrac{(\mu - X)}{\norm{\mu - X}} \\
r^\ast - \norm{\mu - X}\\
\end{pmatrix}
\begin{pmatrix}
\mu - X - r^* \dfrac{(\mu - X)}{\norm{\mu - X}} \\
r^\ast - \norm{\mu - X}\\
\end{pmatrix}^{T}}. 
\end{equation}
\end{thm}
\vspace{1ex}
\def\Abar{\overline{A}}
\def\Zbar{\overline{Z}}

From result (\ref{TLC}) of Theorem \ref{theotlc}, 
we easily derive that
\begin{equation}
\sqrt{n}\,\Sigma^{-1/2} \Gamma_{\theta} 
\pa{\overline{\theta}_{n} - \theta} \tendloi
\mathcal{N}\pa{0, I_{d+1}}.
\end{equation}
Therefore, in order to build confidence balls or statistical tests
for the parameters of the sphere,
matrices $\Gamma_{\theta}$ and
$\Sigma$
must be estimated.

\noindent
Let us decompose $\overline{\theta}_{n}$ under the form $( \Zbar_n , \Abar_n)$ 
where $\Zbar_n \in \bkR^{d}$ estimates the center $\mu$ 
and $\Abar_n \in \Rplus$ the radius $r^\ast$,
and let us denote $U_n := (X_n - \Zbar_n) / \norm{X_n - \Zbar_n}$.
Then we can estimate $\Gamma_{\theta}$ and $\Sigma$ 
by $\widehat\Gamma_n$ and $\widehat\Sigma_n$ iteratively as follows
\begin{align*}
n \widehat\Gamma_{n} & = 
(n-1) \widehat\Gamma_{n-1} + 
\begin{pmatrix}
& \pa{1 - \frac{\Abar_n}{\|X_n - \Zbar_n\|}} I_d + 
\frac{\Abar_n}{\left\| X_n - \Zbar_n \right\|} U_n \, U_n^T & U_n \\
& U_n^T & 1
\end{pmatrix}, \\
n \widehat\Sigma_{n} & = (n-1) \widehat\Sigma_{n-1} + 
\begin{pmatrix}
X_n - \Zbar_n + \Abar_n\,U_n \\
\Abar_n - \| X_n - \Zbar_n\|\end{pmatrix} 
\begin{pmatrix}
X_n - \Zbar_n + \Abar_n\,U_n \\
\Abar_n - \| X_n - \Zbar_n\|\end{pmatrix}^T ,
\end{align*}
where $\widehat\Sigma_1 = I_{d+1}$ and $\widehat\Gamma_1 = I_{d+1}$
to avoid usual problems of invertibility.
It is not hard to show that 
$\widehat\Gamma_{n}$ and
$\widehat\Sigma_{n}$ respectively converge to
$\Gamma_{\theta}$ and $\Sigma$ and then deduce that
\begin{equation}\label{ConvergenceQn}
Q_n := \sqrt{n}\,\widehat\Sigma_{n}^{-1/2} 
\widehat\Gamma_{n} 
\pa{\overline{\theta}_{n} - \theta} \tendloi
\mathcal{N}\pa{0, I_{d+1}}.
\end{equation}
The simulation study of the next section will illustrate the good approximation
of the distribution of $Q_n$ by the standard gaussian for moderate
sample sizes.

\section{Some experiments on simulated data} \label{section_experim_simul}

We study in this section the behavior of the PRM and averaged algorithms on simulated data in the case $d=3$, for small to moderate sample sizes. 
This section first begins with the specification of the compact set involved in the definition of the PRM algorithm which is of course a crucial point. 
We then study the performance of the two algorithms in the case of the whole sphere with the distributions of Examples \ref{ExampleUn} and \ref{ExampleDeux}. 
Finally, we consider the case of the truncated sphere (a half-sphere) 
and we compare our strategy with the one proposed by \cite{Brazey2014}.
\vspace{1ex}

In this simulation study, we shall always consider the same sphere
defined by its center $\mu = (0,0,0)^{T}$ and its radius  $r=50$.
In addition, to reduce sampling effects,  our  results are based on 
$200$ samples of size $n$.
Finally, let us mention that simulations were 
carried out using the statistical software R (see R Core Team, 2013).

\subsection{Choice of the compact set and of the projection}

We discuss here the crucial point of the choice
of the compact set $\calK$ and of the projection $\pi$ 
involved in the definition of the PRM algorithm. 
The main problem is to find a compact set containing the unknown
parameter $\theta$.
We propose to build a preliminary estimation of $\theta$,
using a geometric approach which consists in finding the center and the radius
of a sphere of $\bkR^3$ from 4 non-coplanar distinct points.
We denote by $(\mu_0, r_0)$ this initial estimate of $\theta$.
From this estimate, we define the compact set
$\calK$ by
$\calK := \calK_{\mu_0}\! \times\! \calK_{r_0}$ 
with
$\calK_{\mu_0} := \overline{\mathcal{B}(\mu_0 , r_0 / 10)}$
and $\calK_{r_0} :=  \overline{\mathcal{B}(r_0, r_0/10)}$,
where the choice of the value $r_0 / 10$ for the radius of the balls 
is justified by the discussion about Assumption [A3]
in Section 3.2.
We then define the projector $\pi$ as follows: 
for any $y=(z,a) \in \bkR^3 \times \bkR_+^*$,
we set $\pi(y) := (\pi_{\mu_0}(z), \pi_{r_0}(a))$ with
\begin{equation*}
\pi_{\mu_0}† (z) :=   \left\{\begin{array}{cl}
z & \text{if } z \in \calK_{\mu_0} \\
\mu_0 + \dfrac{r_0}{10} \dfrac{(z - \mu_0)}{\norm{z-\mu_0}} & \text{otherwise }\\
\end{array}\right.
\end{equation*}
and
\begin{equation*}
\pi_{r_0} (a) :=   \left\{\begin{array}{cl}
a & \text{if } a \in \calK_{r_0} \\
r_0 + \dfrac{r_0}{10} \dfrac{(a - r_{0})}{\abs{a-r_{0}}} & \text{otherwise }\\
\end{array}\right.
\end{equation*}
With this strategy, we can raisonnably hope that
if our initial estimate is not too poor, 
then the true parameter belongs to $\calK$
and the quadratic criteria $G$ is convex on $\calK$.
We will see below that even if this preliminary estimation is rough, 
the true parameter belongs to $\calK$ and the PRM algorithm improves
the estimation of $\theta$.
\vspace{1ex}

Let us now describe our strategy to obtain a preliminary estimation of the parameter $\theta = (\mu , r^{\star})$.
Since the data points are spread around the sphere, 
the estimation of the parameters from only one quadruplet
of points is not robust to random fluctuations. 
In order to make the estimation more robust, we consider instead $N$ quadruplets sampled with replacement from the first $K$ points of the sample $X_{1},\ldots,X_n$. 
For each quadruplet, we calculate the center of the sphere which passes through these four points, which gives a sequence of centers 
$\left(\widehat{\mu}_{i}\right)_{1 \leq i \leq N}$. 
The initial estimate of the center, denoted by $\mu_{0}$, is then computed as the median point. 
Finally, we obtain an estimation of the radius by calculating the empirical mean
of the sequence $(\norm{X_i - \mu_{0}})_{1\leq i\leq 50}$.
\vspace{1ex}

A simulation study carried out for various values of $K$ and $N$
in the case of the whole and truncated sphere,
shows that by taking $K=50$ and $N=200$, we obtain
a preliminary estimation of $\theta$ sufficiently good to ensure
that the compact $\calK$ contains $\theta$.
\vspace{1ex}

To close this section, let us mention that although the 
initial estimate is quite accurate, it is necessary to project 
the Robbins-Monro algorithm to ensure the convergence
of the estimator.
Indeed, taking a step sequence of the form 
$\gamma_n = c_\gamma n^{-\alpha}$,
the results given in Table \ref{tabrm}  
show that for some values of $c_\gamma$ and $\alpha$,
the parameter $\theta$ is poorly estimated by
the Robbins-Monro algorithm, while the PRM algorithm (Table \ref{tabprm}) is less sensitive to the step sequence choice.

\begin{table}[H]
\begin{tabular}{|c|c|c|c|c|c|c|}
\cline{3-7} 
\multicolumn{1}{c}{} &  & \multicolumn{1}{c}{} & \multicolumn{1}{c}{} & \multicolumn{1}{c}{$\alpha$} & \multicolumn{1}{c}{} & \tabularnewline
\cline{3-7} 
\multicolumn{1}{c}{} &  & 0.51 & 0.6 & 0.66 & 0.75 & 0.99\tabularnewline
\hline 
 & 1 & 0.27 & 0.14 & 0.09 & 0.06 & 0.23\tabularnewline
\cline{2-7} 
$c_{\gamma} $ & 5 & $10^{8}$ & $10^{6}$ & $10^{5}$ & $10^{4}$
 & $10^{5}$\tabularnewline
\cline{2-7} 
& 10 & $10^{31}$ & $10^{18}$ & $10^{14}$ & $10^{10}$ & $10^{6}$\tabularnewline
\hline 
\end{tabular} 
\caption{Robbins-Monro algorithm.
Errors in quadratic mean of the 200 estimations of 
the center $\mu$ for samples of size $n=2000$ in the case of the 
distribution of Example \ref{ExampleUn}.}
\label{tabrm} 
\end{table}
\begin{table}[H]
\begin{tabular}{|c|c|c|c|c|c|c|}
\cline{3-7} 
\multicolumn{1}{c}{} &  & \multicolumn{1}{c}{} & \multicolumn{1}{c}{} & \multicolumn{1}{c}{$\alpha$} & \multicolumn{1}{c}{} & \tabularnewline
\cline{3-7} 
\multicolumn{1}{c}{} &  & 0.51 & 0.6 & 0.66 & 0.75 & 0.99\tabularnewline
\hline 
 & 1 & 0.28 & 0.15 & 0.09 & 0.05 & 0.24\tabularnewline
\cline{2-7} 
$c_{\gamma}$ & 5 & 1.55 & 0.76 & 0.48 & 0.24 & 0.05\tabularnewline
\cline{2-7} 
 & 10 & 3.22 & 1.35 & 0.94 & 0.43 & 0.08\tabularnewline
\cline{2-7} 
\hline 
\end{tabular}
\caption{PRM algorithm.
Errors in quadratic mean of the 200 estimations of 
the center $\mu$ for samples of size $n=2000$ in the case of the 
distribution of Example \ref{ExampleUn}.}
\label{tabprm}
\end{table}

In the sequel of the simulation study, 
we take a step sequence of the form 
$\gamma_{n}:= n^{-2/3}$ 
($\alpha=2/3$ is often considered as the optimal choice in the literature).

\subsection{Case of the whole sphere}

In what follows, we are interested in the behavior of the PRM and averaged algorithms when samples are distributed on the whole sphere
according to the distribution of Example \ref{ExampleUn} with $\delta=0.1$.
\vspace{1ex}

Figure \ref{boite50rm} shows that the accuracy of the estimations increases with the sample size.
In particular, as expected, the PRM algorithm significantly improves
the initial estimations of the center and the radius (see the first
boxplots which correspond to the initial estimations). 
Moreover, as expected in the case of the "whole sphere", we can see
that the three components of the center $\mu$ are estimated
with the same accuracy.

\begin{figure}[H]
\includegraphics[width=16cm,height=5cm]{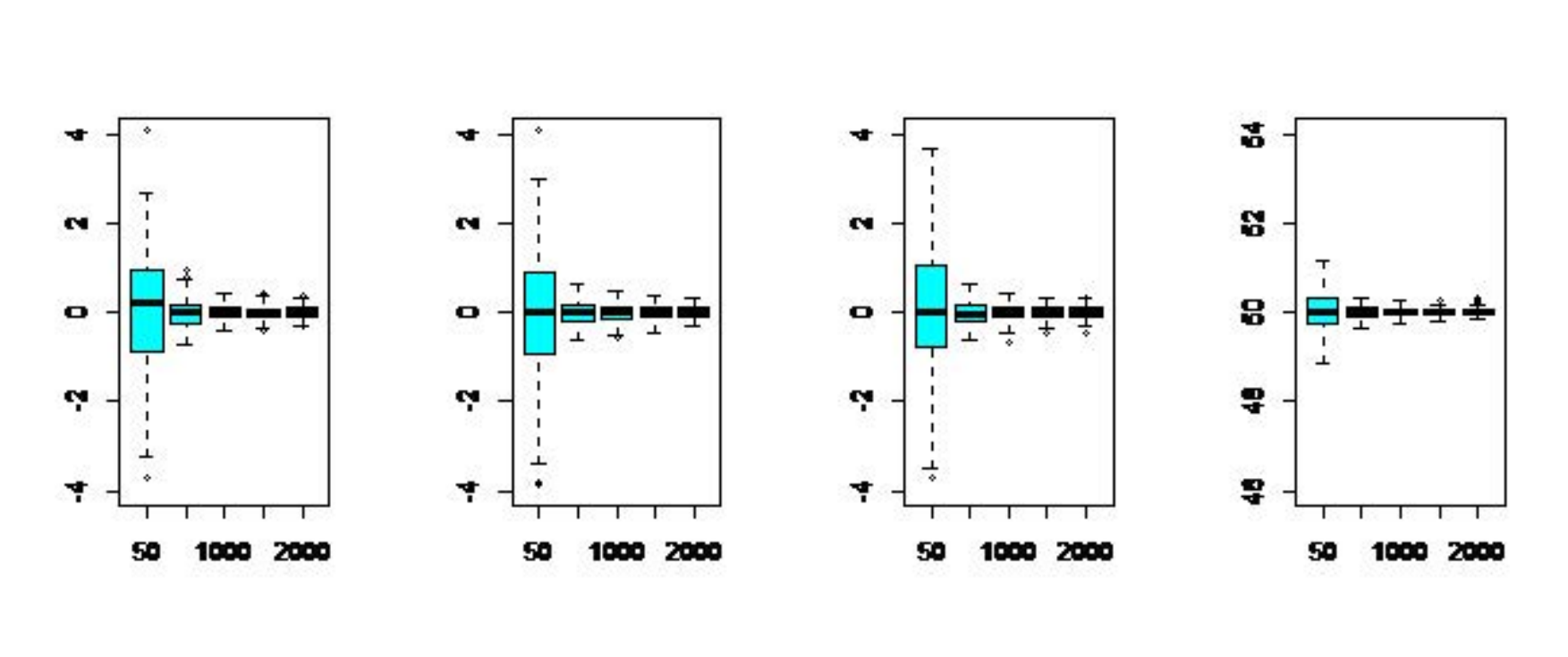}
\caption{Whole sphere with
distribution of Example \ref{ExampleUn}.
From the left to the right, boxplots of estimates of 
$\mu_x , \mu_y , \mu_z$ and $r$ obtained with the PRM algorithm for different sample sizes.}
\label{boite50rm}
\end{figure} 

Let us now examine the gain provided by the use of the averaged
algorithm. 
Figure \ref{boite50av} shows that for small sample sizes, the performances of the two algorithms
are comparable, but when $n$ is greater than $500$, the averaged 
algorithm is more accurate than the PRM algorithm.
We can even think that by forgetting the first estimates
of the PRM algorithm, we improve the behavior of the averaged algoritm
when the sample size is small.

\begin{figure}[H]
\begin{tabular}{cc}
\includegraphics[scale=0.4]{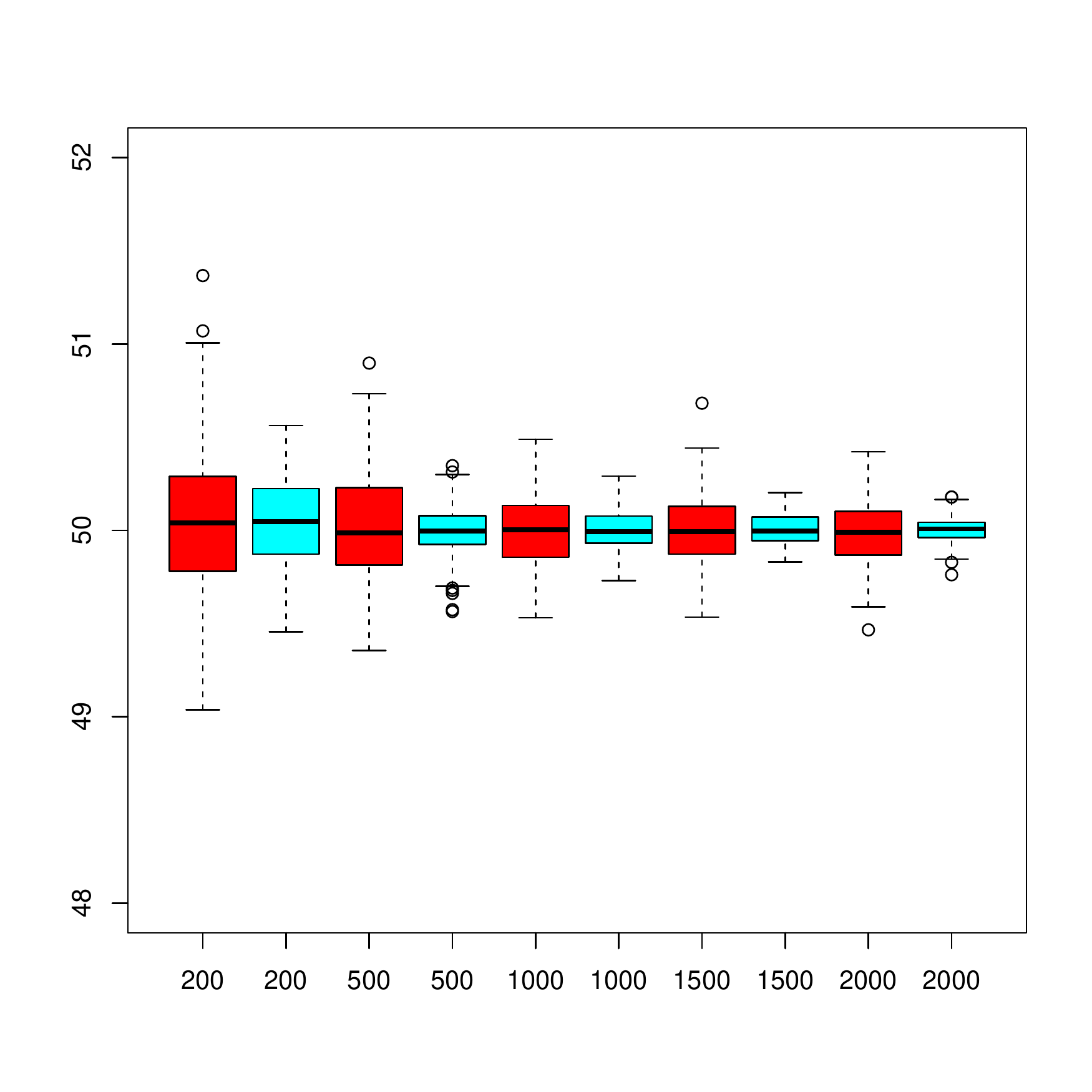}
&
\includegraphics[scale=0.4]{prmavrayon.pdf}
\end{tabular}
\caption{Whole sphere with
distribution of Example \ref{ExampleUn}.
Boxplots of estimates of $\mu_y$ (left) and 
$r$ (right) obtained with the PRM algorithm (in red) and with the averaged algorithm (in blue) for different sample sizes.}
\label{boite50av}
\end{figure}

Finally, let us study the quality
of the Gaussian approximation of the distribution of $Q_n$
for a moderate sample size.
This point is crucial for building confidence intervals or statistical tests for the parameters of the sphere. 

Figure \ref{fignorm} shows that this approximation
is reasonable when $n=2000$. 
Indeed, we can see that the estimated density 
of each component of $Q_n$ is well superimposed with the density of the 
$\mathcal{N}(0,1)$.
To validate these approximations, we perform
a Kolmogorov-Smirnov test at level $5\%$. 
The test enables us to conclude that the normality is not rejected for each component of $Q_n$.

\begin{figure}[H]
\includegraphics[width=16cm,height=5cm]{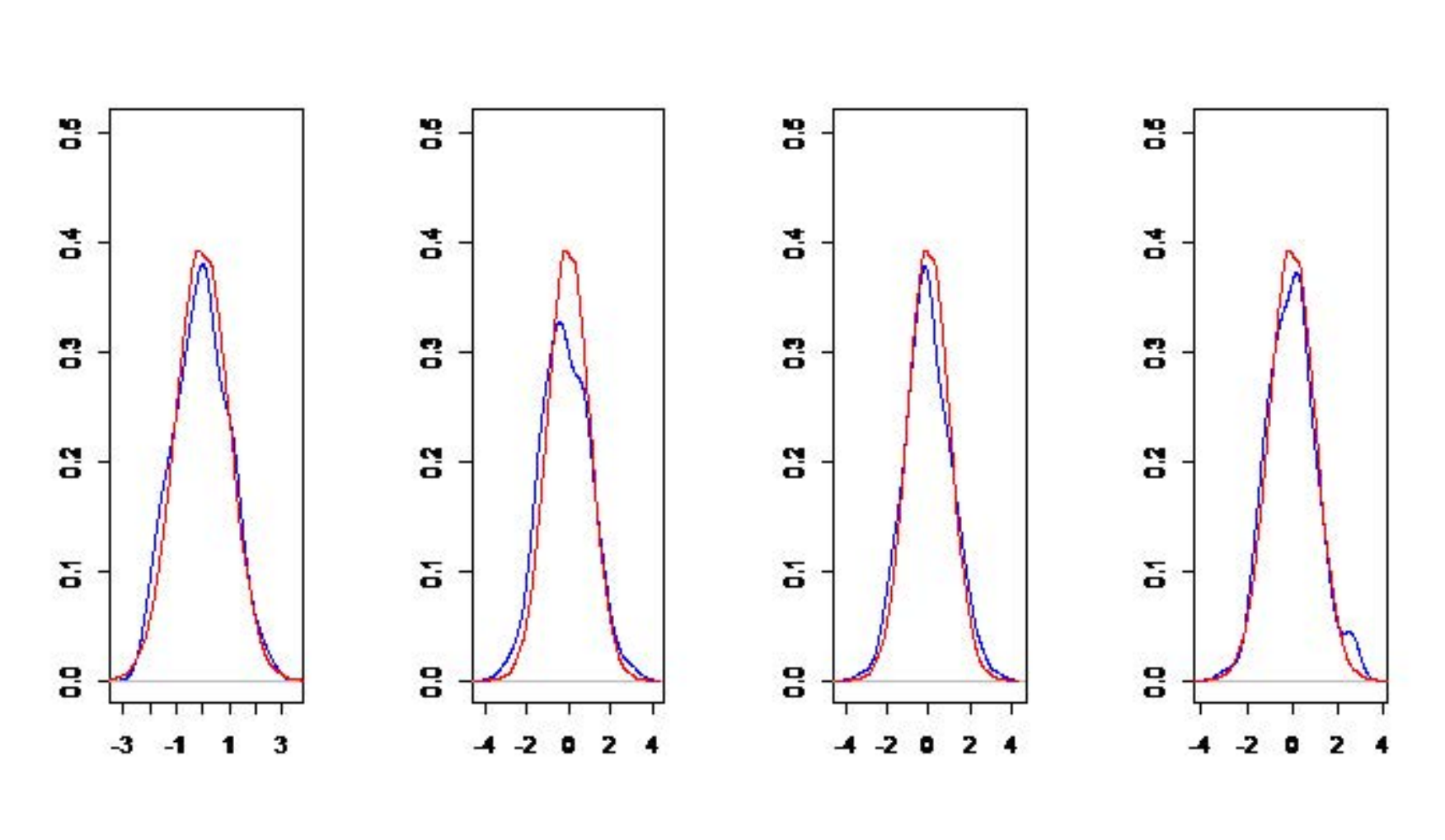}
\caption{
From the left to the right, estimated densities of each components of $Q_{2000}$ superimposed with the standard gaussian density.}
\label{fignorm}
\end{figure}

\subsection{Comparison with 
a backfitting-type algorithm in the case of a half-sphere}

In this section, we compare the performances of the averaged algorithm with the ones of the backfitting algorithm introduced by \cite{Brazey2014}. 
In what follows, we consider samples coming from
the distribution of Example \ref{ExampleDeux}, 
with $\sigma =1$, 
in the case of the half sphere defined
by the set of points whose $y$-component
is positive.
\vspace{1ex}

Results obtained with the two algorithms are presented 
in Figure \ref{denisrm}. 
We focus on parameter
$\mu_y$ for the center since it is the more difficult to estimate. 
We can see that even if the backfitting (BF for short) algorithm is better
than the averaged algorithm, the performances are globally
good, which validates the use of our algorithm for estimating
the parameters of a sphere from 3D-points distributed
around a truncated sphere.
Recall that convergence results are available for our algorithm
in the case of the truncated sphere, contrary to the backfitting algorithm
for which no theoretical result is available in that case.

\begin{figure}[H]
\begin{tabular}{cc}
\includegraphics[scale=0.4]{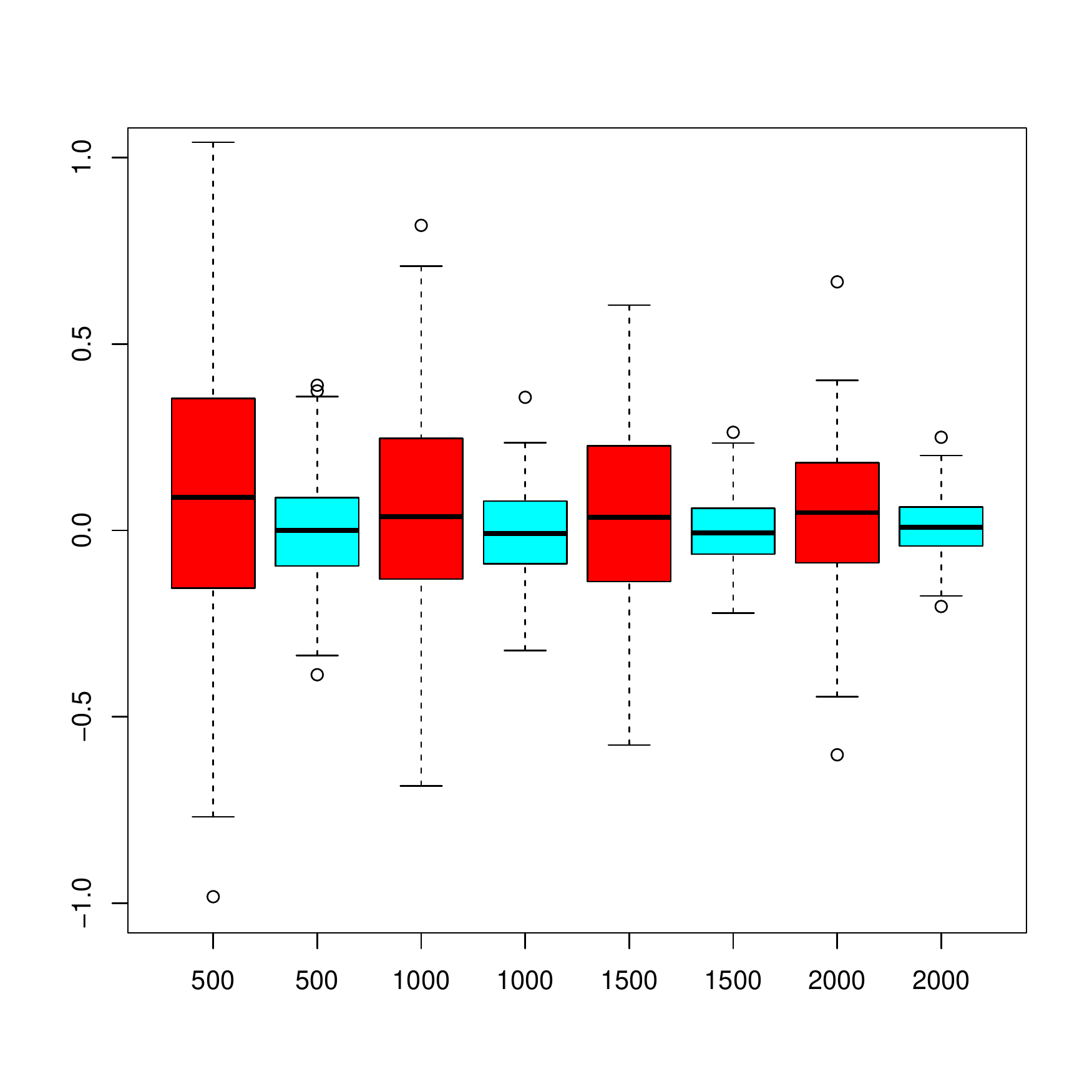} 
&
\includegraphics[scale=0.4]{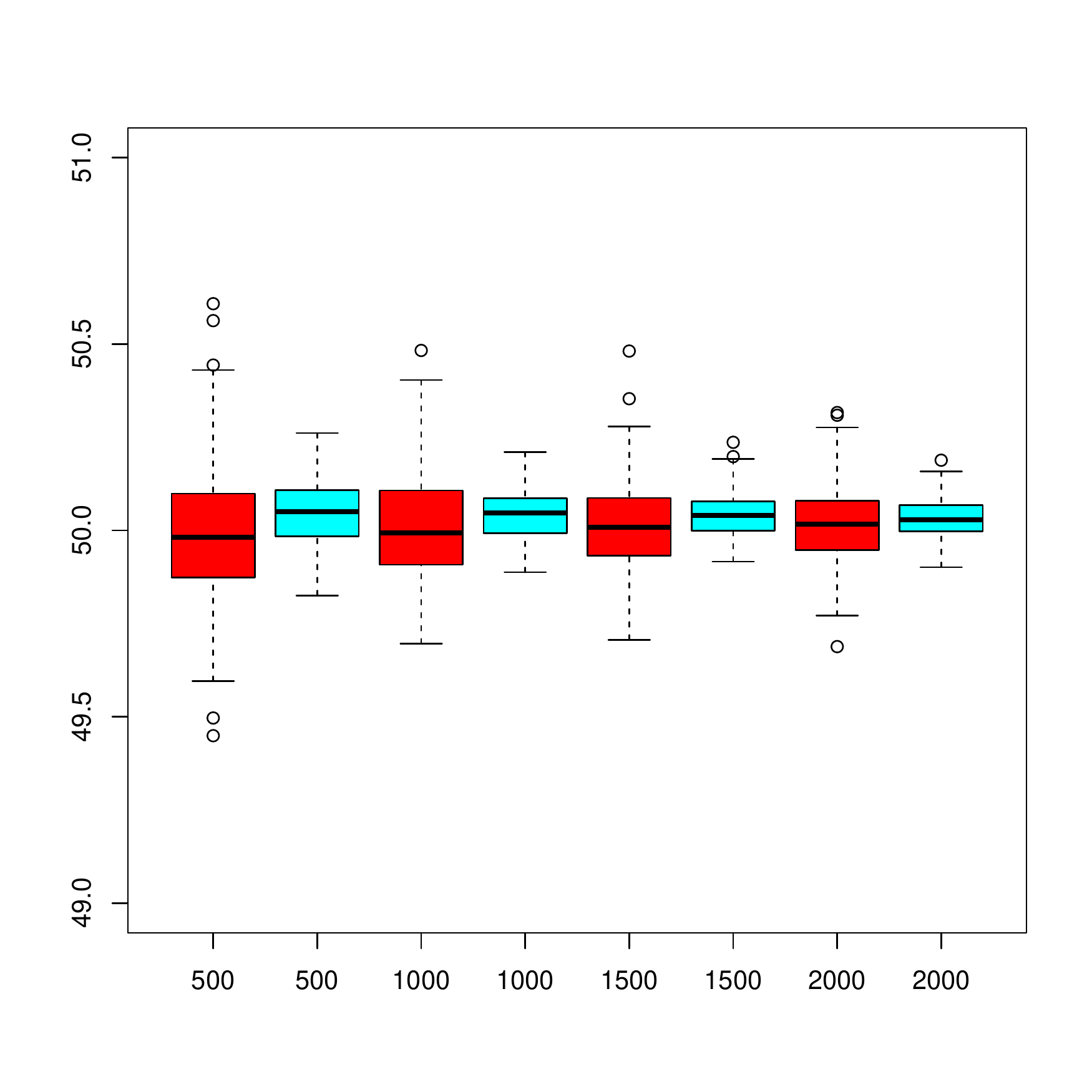} 
\end{tabular}

\caption{Comparison of averaged and BF algorithms.
Boxplots of the estimates of 
$\mu_y$ (on the left) and $r$ (on the right),
obtained with the BF algorithm (in blue) and with the averaged algorithm (in red) for the half sphere in the case of Example \ref{ExampleDeux}.}
 \label{denisrm}
\end{figure}

\section{Conclusion}

We presented in this work a new stochastic algorithm
for estimating the center and the radius of a sphere 
from a sample of points spread around the sphere,
the points being distributed around the complete sphere or
only around a part of the sphere.
\vspace{1ex}

We shown on simulated data that this algorithm
is efficient, less accurate than the backfitting algorithm proposed
in \cite{Brazey2014} but for which no convergence  result
is available for the case of the truncated sphere.
Therefore, our main contribution is to have proposed
an algorithm 
for which we have given asymptotic results such as its strong consistency and its asymptotic normality which can be useful to build confidence balls or statistical tests for example,
as well as non asymptotic results such as the rates of convergence in quadratic mean.
\vspace{1ex}

A possible extension of this work could be to extend
the obtained results to the case of the finite mixture model.
This framework has been considered in \cite{Brazey2014}
but no convergence result is established. 
Proposing a stochastic algorithm for estimating the different parameters of the model 
and obtaining convergence results 
would be a nice challenge.

\medskip

\noindent \textbf{Acknowledgements.} The authors would like to thank Aurélien Vasseur
for his contribution to the start of this work. We also would like to thank Peggy C\'enac and Denis Brazey for their constructive remarks and for their careful reading of the manuscript that allowed to improve the presentation of this work.
Finally, we would like to thank Nicolas Godichon for his help in the creation of Figure 5.

\appendix 
\section{Some convexity results and proof of proposition \ref{propfortconv}}
\label{appendixA}

\renewcommand{\thesection}{\Alph{section}} 
\renewcommand{\theequation}
{\thesection.\arabic{equation}} \setcounter{section}{1}  
\setcounter{equation}{0}
The following lemma ensures that the Matrix in Assumption [A3] is well defined and that the Hessian of $G$ exists for all $y \in \mathbb{R}^{d}\times \mathbb{R}$.
\begin{lem}\label{lemsmallball}
Assume [A1] holds. If $d \geq 3$, there is a positive constant $C$ such that for all $z \in \mathbb{R}^{d}$,
\[
\mathbb{E}\left[ \frac{1}{\left\| X - z \right\| }\right] \leq C.
\]
Moreover, suppose that $W$ admits a bounded density, then for all $d \geq 2$, there is a positive constant $C$ such that for all $z \in \mathbb{R}^{d}$,
\[
\mathbb{E}\left[ \frac{1}{\left\| X - z \right\|}\right] \leq C.
\]
\end{lem}
Note that for the sake of simplicity, we denote by the same way the two constants.
\begin{proof}[Proof of Lemma \ref{lemsmallball}] \textbf{Step 1: $d \geq 3$}

By continuity and applying Assumption \textbf{[A1]}, there are positive constants $\epsilon , C'$ such that for all $z \in \mathcal{B}\left( \mu , \epsilon\right)$,
\[
\mathbb{E}\left[ \frac{1}{\left\| X - z \right\| }\right] \leq C'.
\]
Moreover, let $z \in \mathbb{R}^{d}$ such that $\left\| z-\mu \right\| \geq \epsilon$, we have
\begin{align*}
\mathbb{E}\left[ \frac{1}{\left\| X -z \right\|} \right] &  = \int_{0}^{+\infty} \mathbb{P}\left[ \left\| X - z \right\| \leq \frac{1}{t}\right]dt \\
& = \int_{0}^{M}\mathbb{P}\left[ \left\| X - z \right\| \leq t^{-1}\right]dt + \int_{M}^{\infty}\mathbb{P}\left[ \left\| X - z \right\| \leq t^{-1}\right]dt \\
& \leq M + \int_{M}^{\infty}\mathbb{P}\left[ \left\| X - z \right\| \leq t^{-1}\right]dt,
\end{align*}
with $M$ positive and defined later. Moreover, let $t \geq M$,
\begin{align*}
\mathbb{P}\left[ \left\| X - z \right\| \leq t^{-1} \right] & = \mathbb{P}\left[ \left\| \mu + rWU_{\Omega} - z \right\| \leq t^{-1}\right] \\
& \leq \mathbb{P}\left[ - t^{-1} + \left\| z - \mu \right\| \leq  rW \leq t^{-1} + \left\| z - \mu \right\| , \left(\mu + rWU_{\Omega}\right) \cap \mathcal{B}\left( z , t^{-1} \right) \neq \emptyset \right] ,
\end{align*}
taking $M = \frac{2}{\epsilon}$. With previous condition on $rW$, calculating $\mathbb{P}\left[ \left( \mu + rWU_{\Omega}\right) \cap \mathcal{B}\left( z , t^{-1} \right) \neq \emptyset \right]$ consists in measuring the intersection between a truncated sphere with radius bigger than ~$\epsilon /2$ with a ball of radius $\frac{1}{t}$, with $\frac{1}{t}\leq \frac{\epsilon}{2}$. This is smaller than the surface of the frontier of the ball (see the following figure).

\begin{figure}[H]
\begin{center}
\includegraphics[scale=0.5]{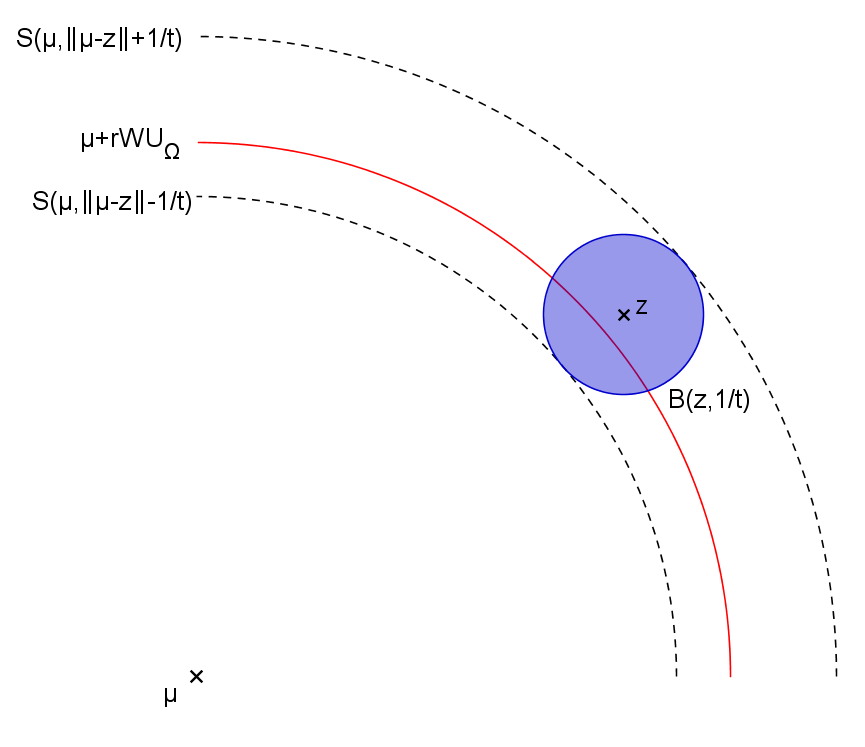}
\end{center}
\caption{}
\label{figannexe}
\end{figure}

Thus, there is a positive constant $k$ such that for all $t \geq M$,
\begin{equation}
\mathbb{P}\left[ \left\| X - z \right\| \leq t^{-1} \right] \leq \frac{k}{t^{d-1}}.
\end{equation}
Finally,
\begin{align*}
\mathbb{E}\left[ \frac{1}{\left\| X-z \right\|}\right] & \leq \frac{2}{\epsilon} + \int_{\frac{2}{\epsilon}}^{+ \infty} k\frac{1}{t^{d-1}}dt \\
& = \frac{2}{\epsilon} + k\frac{\epsilon^{d-2}}{2^{d-2}(d-2)}. 
\end{align*}
We conclude the proof taking $C = \max \left\lbrace C' ,\frac{2}{\epsilon} + k\frac{\epsilon^{d-2}}{2^{d-2}(d-2)} \right\rbrace $.

\bigskip

\textbf{Step 2: $d=2$ and $W$ admits a bounded density} \\
Let $f_{\max}$ be a bound of the density function of $W$. As in previous case, let $z \in \mathbb{R}^{d}$ such that $\left\| z - \mu \right\| \geq \epsilon$, 
\begin{align*}
\mathbb{P}\left[ \left\| X - z \right\| \leq t^{-1} \right] & \leq \mathbb{P}\left[ - t^{-1} + \left\| z - \mu \right\| \leq  rW \leq t^{-1} + \left\| z - \mu \right\| , \left(\mu + rWU_{\Omega}\right) \cap \mathcal{B}\left( z , t^{-1} \right) \neq \emptyset \right] \\
& = \mathbb{P} \left[ \left(\mu + rWU_{\Omega}\right) \cap \mathcal{B}\left( z , t^{-1} \right) \neq \emptyset \Big| - t^{-1} + \left\| z - \mu \right\| \leq  rW \leq t^{-1} + \left\| z - \mu \right\| \right] \\
& \times \mathbb{P}\left[ - t^{-1} + \left\| z - \mu \right\| \leq  rW \leq t^{-1} + \left\| z - \mu \right\| \right] .
\end{align*}
As in previous case, if $t \geq \frac{2}{\epsilon}$, there is a positive constant $k$ such that for all $t \geq \frac{2}{\epsilon}$,
\[
\mathbb{P} \left[ \left(\mu + rWU_{\Omega}\right) \cap \mathcal{B}\left( z , t^{-1} \right) \neq \emptyset \Big| - t^{-1} + \left\| z - \mu \right\| \leq  rW \leq t^{-1} + \left\| z - \mu \right\| \right] \leq k t^{-1}.
\]
Moreover, since $f_{\max}$ is a bound of the density function of $W$,
\begin{align*}
\mathbb{P}\left[ - \frac{1}{t} + \left\| z - \mu \right\| \leq  rW \leq \frac{1}{t} + \left\| z - \mu \right\| \right] & \leq \frac{2rf_{\max}}{t}
\end{align*}
Thus, for all $t \geq \frac{2}{\epsilon}$,
\[
\mathbb{P}\left[ \left\| X - z \right\| \leq t^{-1} \right] \leq \frac{2rf_{\max}k}{t^{2}},
\]
and in a particular case,
\begin{equation}
\mathbb{E}\left[ \frac{1}{\left\| X - z \right\|} \right] \leq \frac{2}{\epsilon} + krf_{\max}\epsilon,
\end{equation}
and one can conclude the proof taking $C= \max \left\lbrace C' , 2\epsilon^{-1} + krf_{\max}\epsilon\right\rbrace$.
\end{proof}

\def\prodscal#1{{\left\langle #1\right\rangle}}

\begin{proof}[Proof of Proposition \ref{propfortconv}]
We want to show there is $c > 0$ such that for any
$y~=~(z,a) ~\in~ \overline{\mathcal{B}(\mu, \varepsilon_\mu)} 
~\times~\overline{\mathcal{B}(r^\ast, \varepsilon_r)}$,
$P(y) := \prodscal{y-\theta\,,\,\Phi(y)} \geq c \norm{y-\theta}$.
We have
\begin{equation}\label{Def_P}
P(y) = P(z,a) = \prodscal{
\begin{pmatrix} z-\mu \\ a - r^{\ast} \end{pmatrix} , 
\begin{pmatrix}
z - \mathbb{E}\cro{X} - a\, \mathbb{E}\left[\dfrac{z - X}{\norm{ z - X}}\right] \\
a - \mathbb{E}\left[ \left\| X - z \right\| \right]
\end{pmatrix}}.
\end{equation}
For any $z\in\bkR^d$, let us set 
$F(z) := \bkE\cro{\norm{X - z}}$
and
$f(z) := \bkE\cro{(z - X) / \norm{z - X}}$.
Note that $f$ is the gradient of $F$.
Using (2.1), we deduce that $F(\mu) = r^\ast$, 
$f(\mu) ~=~ -~ \bkE\cro{U_\Omega}$ and
$\bkE\cro{X} = \mu - r^{\star}f( \mu)$. 
Then, (\ref{Def_P}) can be rewritten as
\begin{align*}
P(y) & = \norm{z - \mu}^{2} + r^{\ast}\prodscal{z - \mu , f(\mu )} -a 
\prodscal{z - \mu , f(z)} + (a-r^{\ast})^{2} - (a - r^{\ast} ) (F(z) - F(\mu)) \\
& = \left\| z - \mu \right\|^{2} - (a-r^{\star})\left\langle z - \mu , f( \mu ) \right\rangle - a \left\langle z - \mu , f(z ) - f( \mu ) \right\rangle + (a-r^{\star})^{2} - (a - r^{\star} ) (F(z) - F(\mu)) .
\end{align*}
Moreover, using the following Taylor's expansions, 
\begin{align*}
& F(z) = F( \mu ) + \langle z - \mu , f(\mu ) \rangle + \frac{1}{2}( z -  \mu )^{T} \nabla f(c) (z-\mu), \\ 
& f(z)  = f(\mu ) + \prodscal{\nabla f (c') ,z-\mu},
\end{align*}
with $c, c' \in [z, \mu]$. We get
\begin{align}
P(y) & = \left\| z - \mu \right\|^{2} - 2(a-r^{\star})\left\langle z - \mu , f(\mu ) \right\rangle - a (z - \mu )^{T}\nabla f (c') (z - \mu )\\
&  + (a -r^{\ast})^{2} - \frac{1}{2}(a-r^{\star})(z- \mu )^{T} \nabla f(c) (z- \mu)
\nonumber
\end{align}
Now, remarking that for any positive constant $A$ and real 
numbers $x,y$, we have 
$2 x y ~\leq ~A ~x^2 ~+ ~y^2 / A$, we derive
\begin{align*}
P(y) & \geq \left\| z - \mu \right\|^{2} - A (a - r^{\ast})^{2} - \frac{1}{A}
\norm{z - \mu}^2 \norm{f(\mu )}^2 - a\, \left\| \nabla f (c) \right\|_{op} \left\| z - \mu \right\|^{2}\\
&  + (a - r^{\ast})^2 -  \frac{1}{2}  \left| a - r^{\star} \right|\left\| \nabla f (c') \right\|_{op} \left\| z - \mu \right\|^{2} . 
\end{align*}
Let us denote by $\displaystyle\lambda_M = 
\sup_{z\in\overline{\mathcal{B}(\mu, \varepsilon_\mu)}}\lambda_{\max} \nabla f (z)$
and choose $A$ such that 
$\norm{f(\mu)}^2~ =~ \norm{\bkE\cro{U_\Omega}}^2~ <~ A ~< ~1$.
Then, for any $z\in\overline{\mathcal{B}(\mu, \varepsilon_\mu)}$
and $a\in\overline{\mathcal{B}(r^\ast, \varepsilon_r)}$, 
we have
\begin{align*}
P(y) & \geq \pa{1 - \frac{1}{A}\norm{f(\mu )}^2 - (r^\ast + \frac{3}{2}\veps_r)\,\lambda_M } \norm{z - \mu}^2 + \pa{1- A} (a - r^{\ast})^2 
\end{align*}
Finally, using assumption \textbf{[A3]}, we close the proof.

\end{proof}

\bigskip

In order to linearize the gradient in the decompositions of the PRM algorithm and get a nice decomposition of the averaged algorithm, we introduce the Hessian matrix of $G$, denoted, for all $y=(z,a) \in \R^{d}\times \R$, by $\Gamma_{y}: \R^{d}\times \R \longrightarrow \R^{d}\times \R$  and defined by :
\begin{equation}\label{eqhess}
\Gamma_{y} = \begin{pmatrix}  
I_d- a \Esp\cro{\dfrac{1}{\| X-z \|}\pa{I_d-\dfrac{(X-z) \otimes (X-z)}{\|X-z\|^2}}}  & & & \Esp\left[\dfrac{X-z}{\|X-z\|}\right]  \\ \\   \Esp\left[\dfrac{X-z}{\|X-z\|}\right]^T  &  & & 1  \end{pmatrix},
\end{equation}
with, for all $z,z',z'' \in \R^{d}$, $z\otimes z' (z'')=\langle z,z'' \rangle z'$. Applying Lemma \ref{lemsmallball}, the Hessian matrix exists for all $y \in \mathbb{R}^{d+1}$.

\begin{prop}\label{majdelta}
Suppose [A1] to [A3] hold, there is a positive constant $C_{\theta}$ such that for all $y \in \mathcal{K}$,
\[
\left\| \Phi (y) - \Gamma_{\theta} \left( y - \theta \right) \right\| \leq C_{\theta} \left\| y - \theta \right\|^{2}.
\]
\end{prop}

\begin{proof}[Proof of Proposition \ref{majdelta}] Under Assumption [A1], by continuity, there are positive constants $C', \epsilon '$ such that for $z \in \mathcal{B}\left( \mu , \epsilon ' \right)$, 
\[
\mathbb{E}\left[ \frac{1}{\left\| X - y \right\|^{2}} \right] \leq C' .
\]
Moreover, note that for all $y \in \mathcal{K}$,
\[
\Phi (y) = \int_{0}^{1}\Gamma_{\theta + t(y - \theta)}(y - \theta)dt.
\]
Thus, with analogous calculus to the ones in the proof of Lemma 5.1 in \cite{CCG2015}, one can check that there is a positive constant $C''$ such that for all $y \in \mathcal{B}\left( \theta , \epsilon ' \right) \cap \mathcal{K}$,
\[
\left\| \Phi (y) - \Gamma_{\theta} \right\| \leq C'' \left\| y - \theta \right\|^{2}.
\]
Moreover, for all $y=(z,a) \in \mathcal{K}$ and $y'=(z',a') \in \mathbb{R}^{d} \times \mathbb{R}$, 
\begin{align*}
 \Gamma_{y}(y')  & =  \begin{pmatrix}
z' - y \mathbb{E}\left[ \frac{1}{\left\| X - z \right\|}\left( z - \frac{\left\langle X- z , z' \right\rangle \left( X - z \right)}{\left\| X - z \right\|^{2}} \right) \right] + a' \mathbb{E}\left[ \frac{X - z}{\left\| X - z \right\|}\right] \\
 \mathbb{E}\left[ \frac{\left\langle X - z , z' \right\rangle}{\left\| X - z \right\|}\right] + a' 
\end{pmatrix} .
\end{align*}
Thus, applying Cauchy-Schwarz's inequality,
\begin{align*}
\left\| \Gamma_{y}(y') \right\|^{2} & = \left\| z' - a \mathbb{E}\left[ \frac{1}{\left\| X - z \right\|}\left( z' - \frac{\left\langle X- z , z' \right\rangle \left( X - z \right)}{\left\| X - z \right\|^{2}} \right) \right] + a' \mathbb{E}\left[ \frac{X - z}{\left\| X - z \right\|}\right] \right\|^{2} \\
& + \left\| \mathbb{E}\left[ \frac{\left\langle X - z , z' \right\rangle}{\left\| X - z \right\|}\right] + a' \right\|^{2} \\
 & \leq 3 \left\| z' \right\|^{2} + 3\left\| a \right\|^{2}\left\| z' \right\|^{2} \mathbb{E}\left[ \frac{1}{\left\| X - z \right\|}\right]^{2} + 3 \left\| a' \right\|^{2} + 2\left\| z' \right\|^{2} + 2\left\| a' \right\|^{2}
\end{align*}
Thus, applying Lemma \ref{lemsmallball}, there are positive constants $A_{1},A_{2}$ such that
\[
\left\| \Gamma_{y}(y') \right\| \leq A_{1} \left\| y' \right\| + A_{2} \left\| y \right\| \left\| y' \right\|
\]
Note that since $\mathcal{K}$ is compact and convex, there is a positive constant $C_{\mathcal{K}}$ such that for all $y \in \mathcal{K}$ and $t \in [0,1]$, $\left\| \theta +t( y - \theta ) \right\| \leq C_{\mathcal{K}}$, and in a particular case,
\[
\left\| \Gamma_{\theta + ( y - \theta )} (y - \theta) \right\| \leq \left( A_{1} + A_{2}C_{\mathcal{K}} \right) \left\| y - \theta \right\| .
\]
Thus, for all $ y \in \mathcal{K}$ such that $\left\| y - \theta \right\| \geq \epsilon '$,
\begin{align*}
\left\| \Phi (y ) - \Gamma_{\theta}\left( y - \theta \right) \right\| & \leq \int_{0}^{1} \left\| \Gamma_{\theta + t( y - \theta ) } \left( y - \theta \right) \right\| dt \\
& \leq \left( A_{1} + A_ {2}C_{\mathcal{K}} \right) \left\| y - \theta \right\| \\
& \leq \frac{1}{\epsilon '} \left( A_{1} + A_{2}C_{\mathcal{K}} \right) \left\| y - \theta \right\|^{2}.
\end{align*}
Thus, we conclude the proof taking $C_{\theta} = \max \left\lbrace C'' , \frac{1}{\epsilon '}\left( A_{1} + A_{2}C_{\mathcal{K}} \right) \right\rbrace$.
\end{proof}

\section{Proof of Section \ref{section_properties}}\label{appendixB}

\begin{proof}[Proof of Theorem \ref{theops}] 
Let us recall that there is a positive constant $c$ such that for all $y \in \mathcal{K}$, $\left\langle \Phi (y) , y - \theta \right\rangle \geq c \left\| y - \theta \right\|^{2}$. The aim is to use previous inequality and the fact that the projection is $1$-lipschitz in order to get an upper bound of $\mathbb{E}\left[ \left\| \widehat{\theta}_{n+1} - \theta \right\|^{2}|\mathcal{F}_{n} \right]$ and apply Robbins-Siegmund theorem to get the almost sure convergence of the algorithm.

\bigskip

\textbf{Almost sure convergence of the algorithm:} Since $\pi$ is $1$-lipschitz,
\begin{align*}
\left\| \widehat{\theta}_{n+1} - \theta \right\|^{2} & = \left\| \pi \left( \widehat{\theta}_{n} - \gamma_{n} \nabla_{y}g \left( X_{n+1} , \widehat{\theta}_{n} \right) \right) - \pi \left( \theta \right) \right\|^{2} \\
& \leq \left\| \widehat{\theta}_{n} - \gamma_{n} \nabla_{y}g \left( X_{n+1} , \widehat{\theta}_{n} \right) - \theta \right\|^{2} \\
& = \left\| \widehat{\theta}_{n} - \theta \right\|^{2} - 2\gamma_{n} \left\langle \nabla_{y}g \left( X_{n+1} , \widehat{\theta}_{n} \right) , \widehat{\theta}_{n} - \theta \right\rangle + \gamma_{n}^{2}\left\| \nabla_{y} g \left( X_{n+1} , \widehat{\theta}_{n} \right) \right\|^{2}
\end{align*} 
Thus, since $\widehat{\theta}_{n}$ is $\mathcal{F}_{n}$-measurable,
\begin{align*}
\mathbb{E}\left[ \left\| \widehat{\theta}_{n+1} - \theta \right\|^{2}|\mathcal{F}_{n} \right] & \leq \left\| \widehat{\theta}_{n} - \theta \right\|^{2} - 2\gamma_{n} \left\langle \mathbb{E}\left[ \nabla_{y}g \left( X_{n+1} , \widehat{\theta}_{n} \right)\Big| \mathcal{F}_{n} \right] , \widehat{\theta}_{n} - \theta \right\rangle + \gamma_{n}^{2}\mathbb{E}\left[ \left\| \nabla_{y} g \left( X_{n+1} , \widehat{\theta}_{n} \right) \right\|^{2}\Big|\mathcal{F}_{n} \right]\\
& = \left\| \widehat{\theta}_{n} - \theta \right\|^{2}  - 2 \gamma_{n} \left\langle \Phi (\widehat{\theta}_{n} ) , \widehat{\theta}_{n} - \theta \right\rangle + \gamma_{n}^{2}\mathbb{E}\left[ \left\| \nabla_{y}g \left( X_{n+1},\widehat{\theta}_{n} \right) \right\|^{2} |\mathcal{F}_{n} \right] \\
& \leq  \left\| \widehat{\theta}_{n} - \theta \right\|^{2}  - 2 c\gamma_{n} \left\| \widehat{\theta}_{n} - \theta \right\|^{2} + \gamma_{n}^{2}\mathbb{E}\left[ \left\| \nabla_{y}g \left( X_{n+1},\widehat{\theta}_{n} \right) \right\|^{2} |\mathcal{F}_{n} \right]
\end{align*}
Moreover, let $\widehat{\theta}_{n} := \left( Z_{n} , A_{n} \right)$ with $Z_{n} \in \mathbb{R}^{d}$ and $A_{n} \in \mathbb{R}$, we have
\begin{align*}
\mathbb{E}\left[ \left\| \nabla_{y} g \left( X_{n+1} , \widehat{\theta}_{n} \right) \right\|^{2} |\mathcal{F}_{n} \right] & =  \mathbb{E}\left[ \left\| Z_{n}-X_{n+1}-A_{n}\frac{Z_{n}-X_{n+1}}{\| Z_{n}-X_{n+1}\|} \right\|^{2}\Big|\mathcal{F}_{n} \right] + \mathbb{E}\left[ \left| A_{n} -  \left\| Z_{n}-X_{n+1} \right\|\right|^{2}|\mathcal{F}_{n} \right] \\
& \leq 4 \mathbb{E}\left[ \left\| Z_{n}-X_{n+1} \right\|^{2}\big|\mathcal{F}_{n} \right] + 4 \left( A_{n} \right)^{2}   \\
& \leq 8 \| Z_{n}- \mu \|^{2} + 8(r^{\star})^{2} + 8\left(A_{n}-r^{\star}\right)^{2} + 8 \mathbb{E}\left[ \| \mu -X_{n+1} \|^{2}|\mathcal{F}_{n} \right] \\
& = 8\left\| \widehat{\theta}_{n}-\theta \right\|^{2}+ 8(r^{\star})^{2} + 8 r^{2}\mathbb{E}\left[ W^{2}\right] .
\end{align*}
Let $M := 8(r^{\star})^{2} + 8 r^{2}\mathbb{E}\left[ W^{2}\right]$, we have
\begin{equation}\label{majord2}
\mathbb{E}\left[ \left\| \widehat{\theta}_{n} - \theta \right\|^{2}|\mathcal{F}_{n} \right] \leq \left( 1+8\gamma_{n}^{2} \right)\left\| \widehat{\theta}_{n} - \theta \right\|^{2} -2c\gamma_{n} \left\| \widehat{\theta}_{n} - \theta \right\|^{2} + \gamma_{n}^{2}M.
\end{equation}
Applying Robbins-Siegmund's theorem (see \cite{DUF1997} for instance), $\left\| \widehat{\theta}_{n} - \theta \right\|^{2}$ converges almost surely to a finite random variable, and in a particular case,
\[
\sum_{k=1}^{\infty} \gamma_{k} \left\| \widehat{\theta}_{k} - \theta \right\|^{2} < +\infty .
\]
Thus, since $\sum_{k \geq 1}\gamma_{k} = + \infty$,
\begin{equation}
\lim_{n \rightarrow + \infty}\left\| \widehat{\theta}_{n} - \theta \right\|^{2} = 0 \quad a.s. 
\end{equation}

\medskip

\textbf{Number of times the projection is used}

Let $N_{n}:= \sum_{k=1}^{n} \mathbf{1}_{\left\lbrace \widehat{\theta}_{k} - \gamma_{k} \nabla_{y} \left( X_{k+1} , \widehat{\theta}_{k} \right) \notin \mathcal{K} \right\rbrace }$. This sequence is non-decreasing, and suppose by contradiction that $N_{n}$ goes to infinity. Thus, there is a subsequence $\left( n_{k} \right)$ such that $\left( N_{n_{k}} \right)$ is increasing, i.e for all $k \geq 1$, $\widehat{\theta}_{n_{k}} - \gamma_{n} \nabla_{y} g \left( X_{n_{k}+1} , \widehat{\theta}_{n_{k}} \right) \notin \mathcal{K}$, and in a particular case, $\widehat{\theta}_{n_{k}+1} ~ \in ~ \partial \mathcal{K}$, where $\partial \mathcal{K}$ is the frontier of $\mathcal{K}$. Let us recall that $\theta$ is in the interior of $\mathcal{K}$, i.e let $d_{\min}:=\inf_{y \in \partial \mathcal{K}}\left\| \theta - y \right\| $, we have $d_{\min}>0$. Thus,
\[
\left\| \widehat{\theta}_{n_{k}+1} - \theta \right\| \geq d_{\min} \quad a.s,
\]
and,
\[
\lim_{k \rightarrow \infty}\left\| \widehat{\theta}_{n_{k}+1} - \theta \right\| = 0 \geq d_{min}> 0 \quad a.s,
\]
which leads to a contradiction.
\end{proof}

\begin{proof}[Proof of Theorem \ref{theol2l4}]
\textbf{Convergence in quadratic mean}

The aim is to obtain an induction relation for the quadratic mean error. Let us recall inequality (\ref{majord2}),
\[
\mathbb{E}\left[ \left\| \widehat{\theta}_{n+1} - \theta \right\|^{2}|\mathcal{F}_{n} \right] \leq \left( 1+8\gamma_{n}^{2} \right)\left\| \widehat{\theta}_{n} - \theta \right\|^{2} -2c\gamma_{n} \left\| \widehat{\theta}_{n} - \theta \right\|^{2} + \gamma_{n}^{2}M.
\]
Then we have
\[
\mathbb{E}\left[ \left\| \widehat{\theta}_{n+1} - \theta \right\|^{2} \right] \leq \left( 1-c\gamma_{n} + 8\gamma_{n}^{2} \right) \mathbb{E}\left[ \left\| \widehat{\theta}_{n} - \theta \right\|^{2}\right] + M \gamma_{n}^{2},
\]
and one can conclude the proof with the help of an induction (see \cite{godichon2015} for instance) or applying a lemma of stabilization (see \cite{duflo1996algorithmes}).

\bigskip

\textbf{$L^{p}$ rates of convergence}\\
Let $p \geq 2$, we now prove with the help of a strong induction that for all integer $p' \leq p$, there is a positive constant $C_{p'}$ such that for all $n \geq 1$,
\[
\mathbb{E}\left[ \left\| \widehat{\theta}_{n} - \theta \right\|^{2p'} \right] \leq \frac{C_{p'}}{n^{p'\alpha}}.
\]
This inequality is already checked for $p'=1$. Let $p' \geq 2$, we suppose from now that for all integer $k < p'$, there is a positive constant $C_{k}$ such that for all $n \geq 1$,
\[
\mathbb{E}\left[ \left\| \widehat{\theta}_{n} - \theta \right\|^{2k} \right] \leq \frac{C_{k}}{n^{k\alpha}} .
\]
We now search to give an induction relation for the $L^{2p'}$-error. Let us recall that
\[
\left\| \widehat{\theta}_{n+1} - \theta \right\|^{2} \leq \left\| \widehat{\theta}_{n} - \theta \right\|^{2} - 2 \gamma_{n} \left\langle \nabla_{y}g\left( X_{n+1},\widehat{\theta}_{n} \right) , \widehat{\theta}_{n} - \theta \right\rangle + \gamma_{n}^{2}\left\|  \nabla_{y}g\left( X_{n+1} , \widehat{\theta}_{n} \right)\right\|^{2}.
\]
We suppose from now that $\mathbb{E}\left[ W^{2p} \right] < + \infty $ (and in a particular case, $\mathbb{E}\left[ W^{k} \right] < + \infty$  for all integer $k \leq 2p$) and let $U_{n+1} := \nabla_{y}g\left( X_{n+1} , \widehat{\theta}_{n} \right)$. We have
\begin{align}
\notag \left\| \widehat{\theta}_{n+1} - \theta \right\|^{2p'} & \leq \left(\left\| \widehat{\theta}_{n} - \theta \right\|^{2} + \gamma_{n}^{2} \left\| U_{n+1} \right\|^{2} \right)^{p'} - 2 p'\gamma_{n}\left\langle \widehat{\theta}_{n} - \theta , U_{n+1} \right\rangle \left( \left\| \widehat{\theta}_{n} - \theta \right\|^{2} + \gamma_{n}^{2} \left\| U_{n+1 } \right\|^{2} \right)^{p'-1} \\
\label{momordp}& + \sum_{k=2}^{p'}\binom{p'}{k}\gamma_{n}^{k}\left| \left\langle \widehat{\theta}_{n} - \theta , U_{n+1} \right\rangle \right|^{k} \left( \left\| \widehat{\theta}_{n} - \theta \right\|^{2} + \gamma_{n}^{2} \left\| U_{n+1} \right\|^{2} \right)^{p'-k}.
\end{align}
The aim is to bound each term on the right-hand side of previous inequality. In this purpose, we first need to introduce some technical inequalities.
\begin{align*}
\left\| U_{n+1} \right\|^{2} & \leq 2 \left\| \nabla_{y}g\left( X_{n+1} , \widehat{\theta}_{n} \right) \right\|^{2} + 2 \mathbb{E}\left[ \left\| \nabla_{y}g\left( X_{n+1} , \widehat{\theta}_{n} \right) \right\|^{2} \Big| \mathcal{F}_{n} \right]     \\
& \leq 16 \left( 2\left\| \widehat{\theta}_{n} - \theta \right\|^{2} + 2\left( r^{\star} \right)^{2} + \left\| \mu - X_{n+1} \right\|^{2} + r^{2}\mathbb{E}\left[ W^{2} \right] \right)  . 
\end{align*}
Thus, applying Lemma A.1 in \cite{godichon2015} for instance, for all integer $k \leq p'$,
\[
 \left\| U_{n+1} \right\|^{2k} \leq 4^{k-1}16^{k} \left( 2^{k}\left\| \widehat{\theta}_{n} - \theta \right\|^{2k} + 2^{k}\left( r^{\star} \right)^{2k} + \left\| X_{n+1} - \mu \right\|^{2k} + r^{2k}\left(\mathbb{E}\left[ W^{2} \right]\right)^{k} \right) .
\]
In a particular case, since for all $k \leq p$, $\mathbb{E}\left[ W ^{2k} \right] < +\infty$, there are positive constants $A_{1,k},A_{2,k}$ such that for all $n \geq 1$,
\begin{align}
\notag \mathbb{E}\left[ \left\| U_{n+1} \right\|^{2k} \Big| \mathcal{F}_{n}\right] & \leq 4^{3k-1}\left( 2^{k}\left\| \widehat{\theta}_{n}-\theta \right\|^{2k} + 2^{k}\left( r^{\star} \right)^{2k} + r^{2k} \mathbb{E}\left[ W^{2k} \right] + r^{2k}\left( \mathbb{E}\left[ W^{2} \right] \right)^{k} \right) \\
\label{majun} & \leq A_{1,k}\left\| \widehat{\theta}_{n} - \theta \right\|^{2k} + A_{2,k}. 
\end{align}
We can now bound the expectation of the three terms on the right-hand side of inequality ~ (\ref{momordp}). First, since $\widehat{\theta}_{n}$ is $\mathcal{F}_{n}$- measurable, applying inequality (\ref{majun}), let
\begin{align*}
( * ) & := \mathbb{E}\left[ \left(\left\| \widehat{\theta}_{n} - \theta \right\| + \gamma_{n}^{2} \left\| U_{n+1} \right\|^{2} \right)^{p'} \right] \\
& = \mathbb{E}\left[ \left\| \widehat{\theta}_{n} - \theta \right\|^{2p'}\right] + \sum_{k=1}^{p'}\binom{p'}{k}\gamma_{n}^{2k}\mathbb{E}\left[ \left\| U_{n+1} \right\|^{2k} \left\| \widehat{\theta}_{n} - \theta \right\|^{2p'-2k}\right] \\
& \leq \mathbb{E}\left[ \left\| \widehat{\theta}_{n} - \theta \right\|^{2p'}\right] + \sum_{k=1}^{p'}\binom{p'}{k}\gamma_{n}^{2k}\mathbb{E}\left[\left( A_{1,k}\left\| \widehat{\theta}_{n} - \theta \right\|^{2k} + A_{2,k} \right) \left\| \widehat{\theta}_{n} - \theta \right\|^{2p'-2k}\right]
\end{align*}
Let $B:= \sum_{k=1}^{p'}c_{\gamma}^{2k-2}A_{1,k}$, using previous inequality and by induction, 
\begin{align}
\notag (*) &  \leq \left( 1+B\gamma_{n}^{2} \right) \mathbb{E}\left[ \left\| \widehat{\theta}_{n} - \theta \right\|^{2p'} \right] + \sum_{k=1}^{p'} \binom{p'}{k} \gamma_{n}^{2k}A_{2,k}\mathbb{E}\left[ \left\| \widehat{\theta}_{n} - \theta \right\|^{2p'-2k} \right] \\
\notag & \leq \left( 1+B\gamma_{n}^{2} \right) \mathbb{E}\left[ \left\| \widehat{\theta}_{n} - \theta \right\|^{2p'} \right] + \sum_{k=1}^{p'} \binom{p'}{k} c_{\gamma}^{2k}A_{2,k}\frac{C_{k}}{n^{(p'+k)\alpha}} \\
\label{maj*} & \leq \left( 1+B\gamma_{n}^{2} \right) \mathbb{E}\left[ \left\| \widehat{\theta}_{n} - \theta \right\|^{2p'} \right] + O \left( \gamma_{n}^{p'+1} \right) .
\end{align}
In the same way, applying Cauchy-Schwarz's inequality, let
\begin{align*}
(**) & := -2p'\gamma_{n}\mathbb{E}\left[  \left\langle \widehat{\theta}_{n} - \theta , U_{n+1} \right\rangle \left( \left\| \widehat{\theta}_{n} - \theta \right\|^{2} + \gamma_{n}^{2} \left\| U_{n+1} \right\|^{2} \right)^{p'-1}\right] \\
& \leq  -2p' \gamma_{n} \mathbb{E}\left[ \left\langle \widehat{\theta}_{n}  - \theta , U_{n+1} \right\rangle \left\| \widehat{\theta}_{n} - \theta \right\|^{2p'-2}\right] \\
& + 2p' \gamma_{n} \mathbb{E}\left[\left\| \widehat{\theta}_{n} - \theta \right\| \left\| U_{n+1} \right\| \sum_{k=1}^{p'-1}\binom{p'-1}{k}\gamma_{n}^{2k} \left\| U_{n+1} \right\|^{2k} \left\| \widehat{\theta}_{n} - \theta \right\|^{2p'-2k} \right] .
\end{align*}
Moreover, since $\widehat{\theta}_{n}$ is $\mathcal{F}_{n}$-measurable, applying Proposition \ref{propfortconv},
\begin{align*}
-2p' \gamma_{n} \mathbb{E}\left[ \left\langle \widehat{\theta}_{n}  - \theta , U_{n+1} \right\rangle \left\| \widehat{\theta}_{n} - \theta \right\|^{2p'-2}\right] & = -2p'\gamma_{n} \mathbb{E}\left[ \left\langle \widehat{\theta}_{n} - \theta , \mathbb{E}\left[ U_{n+1} |\mathcal{F}_{n} \right] \right\rangle \left\| \widehat{\theta}_{n} - \theta \right\|^{2p'-2} \right]  \\
& = -2p'\gamma_{n}\mathbb{E}\left[ \left\langle \widehat{\theta}_{n} - \theta , \Phi (\widehat{\theta}_{n}) \right\rangle \left\| \widehat{\theta}_{n} - \theta \right\|^{2p'-2} \right] \\
& \leq -2p'c \gamma_{n} \mathbb{E}\left[ \left\| \widehat{\theta}_{n} - \theta \right\|^{2p'} \right] .
\end{align*}
Moreover, since $2ab \leq a^{2} + b^{2}$, let
\begin{align*}
(**') & := 2p' \gamma_{n} \mathbb{E}\left[\left\| \widehat{\theta}_{n} - \theta \right\| \left\| U_{n+1} \right\| \sum_{k=1}^{p'-1}\binom{p'-1}{k}\gamma_{n}^{2k} \left\| U_{n+1} \right\|^{2k} \left\| \widehat{\theta}_{n} - \theta \right\|^{2p'-2k} \right] \\
& \leq p'\gamma_{n} \mathbb{E}\left[ \left( \left\| \widehat{\theta}_{n} - \theta \right\|^{2} + \left\| U_{n+1} \right\|^{2} \right)\sum_{k=1}^{p'-1}\binom{p'-1}{k}\gamma_{n}^{2k} \left\| U_{n+1} \right\|^{2k} \left\| \widehat{\theta}_{n} - \theta \right\|^{2p'-2k} \right] \\
& \leq p'\gamma_{n} \sum_{k=1}^{p'-1}\binom{p'-1}{k}\gamma_{n}^{2k}\left( \mathbb{E}\left[ \left\| U_{n+1} \right\|^{2k+2}\left\| \widehat{\theta}_{n}  - \theta \right\|^{2p'-2k} \right] + \mathbb{E}\left[ \left\| U_{n+1} \right\|^{2k}\left\| \widehat{\theta}_{n} - \theta \right\|^{2p'+2-2k} \right] \right) .
\end{align*}
With analogous calculus to the ones for inequality (\ref{maj*}), one can check that there is a positive constant $B'$ such that for all $n \geq 1$, 
\[ (**')\leq B'\gamma_{n}^{2}\mathbb{E}\left[ \left\| \widehat{\theta}_{n} - \theta \right\|^{2p'} \right] +  O \left( \gamma_{n}^{(p'+1)\alpha} \right).
\]
Thus, 
\begin{align}
\label{maj**} -2\gamma_{n}\mathbb{E}\left[ \gamma_{n} \left\langle \widehat{\theta}_{n} - \theta , U_{n+1} \right\rangle \left( \left\| \widehat{\theta}_{n} - \theta \right\|^{2} + \gamma_{n}^{2} \left\| U_{n+1} \right\|^{2} \right)^{p'-1}\right] & \leq \left( -2cp'\gamma_{n} + B'\gamma_{n}^{2} \right) \mathbb{E}\left[ \left\| \widehat{\theta}_{n} - \theta \right\|^{2p'} \right] \\
\notag &   + O \left( \gamma_{n}^{(p'+1)\alpha} \right) .
\end{align} 
Finally, applying Lemma A.1 in \cite{godichon2015} and since $\left| \left\langle a ,b \right\rangle\right| \leq \frac{1}{2}\left\| a \right\|^{2}+\frac{1}{2}\left\| b \right\|^{2}$, let
\begin{align*}
(***) & := \sum_{k=2}^{p'}\binom{p'}{k}\gamma_{n}^{k}\mathbb{E}\left[ \left| \left\langle \widehat{\theta}_{n} - \theta , U_{n+1} \right\rangle \right|^{k} \left( \left\| \widehat{\theta}_{n} - \theta \right\|^{2} + \gamma_{n}^{2} \left\| U_{n+1} \right\|^{2} \right)^{p'-k} \right] \\
& \leq \sum_{k=2}^{p'}\binom{p'}{k}\gamma_{n}^{k}\mathbb{E}\left[ \left( \frac{1}{2}\left\| \widehat{\theta}_{n} - \theta \right\|^{2} + \frac{1}{2} \left\| U_{n+1} \right\|^{2} \right)^{k} \left( \left\| \widehat{\theta}_{n} - \theta \right\|^{2} + \gamma_{n}^{2} \left\| U_{n+1} \right\|^{2} \right)^{p'-k} \right]  \\
& \leq \sum_{k=2}^{p'}\binom{p'}{k}2^{p'-k-2}\gamma_{n}^{k}\mathbb{E}\left[ \left( \left\| \widehat{\theta}_{n} - \theta \right\|^{2k} +  \left\| U_{n+1} \right\|^{2k}  \right)\left( \left\| \widehat{\theta}_{n} - \theta \right\|^{2p'-2k} + \gamma_{n}^{2p'-2k} \left\| U_{n+1} \right\|^{2p'-2k} \right) \right] 
\end{align*}
Thus, with analogous calculus to the ones for inequality (\ref{maj*}), one can check that there is a positive constant $B''$ such that for all $n \geq 1$,
\begin{equation}
\label{maj***} \sum_{k=2}^{p'}\binom{p'}{k}\gamma_{n}^{k}\mathbb{E}\left[ \left| \left\langle \widehat{\theta}_{n} - \theta , U_{n+1} \right\rangle \right|^{k} \left( \left\| \widehat{\theta}_{n} - \theta \right\|^{2} + \gamma_{n}^{2} \left\| U_{n+1} \right\|^{2} \right)^{p'-k} \right]  \leq B'' \gamma_{n}^{2}\mathbb{E}\left[ \left\| \widehat{\theta}_{n} - \theta \right\|^{2p'}\right] + O \left( \gamma_{n}^{p'+1} \right) .
\end{equation}
Finally, applying inequalities (\ref{maj*}) to (\ref{maj***}), there are positive constants $B_{1},B_{2}$ such that for all $n \geq 1$,
\begin{equation}
\mathbb{E}\left[ \left\| \widehat{\theta}_{n+1} - \theta \right\|^{2p'} \right] \leq \left( 1-2p'c\gamma_{n} + B_{1}\gamma_{n}^{2} \right) \mathbb{E}\left[ \left\| \widehat{\theta}_{n} - \theta \right\|^{2p'} \right] + B_{2}\gamma_{n}^{p'+1}.
\end{equation}
Thus, with the help of an induction on $n$ or applying a lemma of stabilization (see \cite{duflo1996algorithmes} for instance), one can check that there is a positive constant $C_{p'}$ such that for all $n\geq 1$,
\[
\mathbb{E}\left[ \left\| \widehat{\theta}_{n} - \theta \right\|^{2p'}\right] \leq \frac{C_{p'}}{n^{p'\alpha}},
\]
which concludes the induction on $p'$ and the proof.

\bigskip

\textbf{Bounding $\mathbb{P}\left[ \widehat{\theta}_{n} - \gamma_{n}\nabla_{y}g\left( X_{n+1} , \widehat{\theta}_{n} \right) \notin \mathcal{K}\right]$}\\
Let us recall that $d_{\min} = \inf_{y \in \partial \mathcal{K} }\left\| y - \theta \right\| > 0$ and that if $W$ admits a $2p$-th moment, there is a positive constant $C_{p}$ such that for all $n\geq 1$, $\mathbb{E}\left[ \left\| \widehat{\theta}_{n} - \theta \right\|^{2p}\right] \leq \frac{C_{p}}{n^{p\alpha}}$. Thus, for all $n \geq 1$,
\begin{align*}
\frac{C_{p}}{(n+1)^{p\alpha}} & \geq \mathbb{E}\left[ \left\| \widehat{\theta}_{n+1} - \theta \right\|^{2p} \right] \\
& \geq \mathbb{E}\left[ \left\| \widehat{\theta}_{n+1} - \theta \right\|^{2p} \mathbf{1}_{\left\lbrace  \widehat{\theta}_{n} - \gamma_{n}\nabla_{y}g\left( X_{n+1} , \widehat{\theta}_{n} \right) \notin \mathcal{K} \right\rbrace } \right] \\
& \geq d_{\min}^{2p}\mathbb{P}\left[ \widehat{\theta}_{n} - \gamma_{n}\nabla_{y}g\left( X_{n+1} , \widehat{\theta}_{n} \right) \notin \mathcal{K} \right] .
\end{align*}
Finally,
\[
\mathbb{P}\left[ \widehat{\theta}_{n} - \gamma_{n}\nabla_{y}g\left( X_{n+1} , \widehat{\theta}_{n} \right) \notin \mathcal{K}\right] \leq \frac{C_{p}}{d_{\min}^{2p}}\frac{1}{(n+1)^{p\alpha}} \leq \frac{C_{p}}{d_{\min}^{2p}}\frac{1}{n^{p\alpha}} .
\]

\end{proof}

\begin{proof}[Proof of Theorem \ref{theol2moy}] The aim is, in a first time, to exhibit a nice decomposition of the averaged algorithm. In this purpose, let us introduce this new decomposition of the PRM algorithm
\begin{equation}
\label{decxi} \widehat{\theta}_{n+1} - \theta = \widehat{\theta}_{n} - \theta - \gamma_{n} \Phi \left( \widehat{\theta}_{n} \right) + \gamma_{n} \xi_{n+1} + r_{n},
\end{equation}
with 
\begin{align*}
\xi_{n+1} & := -\nabla_{y}g\left(X_{n+1},\widehat{\theta}_{n} \right) + \Phi \left(\widehat{\theta}_{n} \right) ,  \\
r_{n} & := \pi \left( \widehat{\theta}_{n} - \gamma_{n} \nabla_{y}g \left( X_{n+1} , \widehat{\theta}_{n} \right)\right) - \widehat{\theta}_{n} + \gamma_{n} \nabla_{y}g \left( X_{n+1} , \widehat{\theta}_{n} \right)  .
\end{align*}
Remark that $\left( \xi_{n} \right)$ is a sequence of martingale differences adapted to the filtration $\left( \mathcal{F}_{n} \right)$ and $r_{n}$ is equal to $0$ when $\widehat{\theta}_{n} - \gamma_{n} \nabla_{y}g \left( X_{n+1} , \widehat{\theta}_{n} \right) \in \mathcal{K}$. Moreover, linearizing the gradient, decomposition (\ref{decxi}) can be written as
\begin{align}
\label{decdelta} \widehat{\theta}_{n+1} - \theta & = \left( I_{\mathbb{R}^{d} \times \mathbb{R}} - \gamma_{n}\Gamma_{\theta} \right) \left( \widehat{\theta}_{n} - \theta \right) + \gamma_{n} \xi_{n+1} - \gamma_{n} \delta_{n} + r_{n},
\end{align}
where $\delta_{n} := \Phi \left( \widehat{\theta}_{n} \right) - \Gamma_{\theta}\left( \widehat{\theta}_{n} - \theta \right)$ is the remainder term in the Taylor's expansion of the gradient. This can also be decomposed as
\[
\Gamma_{\theta} \left( \widehat{\theta}_{n} - \theta \right) = \frac{\widehat{\theta}_{n} - \theta}{\gamma_{n}} - \frac{\widehat{\theta}_{n+1} - \theta}{\gamma_{n}} - \delta_{n} + \frac{r_{n}}{\gamma_{n}} + \xi_{n+1}.
\]
As in \cite{pelletier2000asymptotic}, summing these equalities, applying Abel's transform and dividing by $n$,
\begin{align}
 \Gamma_{\theta} \left( \overline{\theta}_{n} - \theta \right) &  = \frac{1}{n}\left(\frac{\widehat{\theta}_{1} - \theta}{\gamma_{1}} - \frac{\widehat{\theta}_{n+1} - \theta}{\gamma_{n}} +  \sum_{k=2}^{n}\left( \frac{1}{\gamma_{k}} - \frac{1}{\gamma_{k-1}}\right)\left( \widehat{\theta}_{k} - \theta \right)  - \sum_{k=1}^{n} \delta_{k} + \sum_{k=1}^{n} \frac{r_{k}}{\gamma_{k}} \right)  + \frac{1}{n}\sum_{k=1}^{n}\xi_{k+1}.
\end{align}
We now give the rate of convergence in quadratic mean of each term using Theorem~ \ref{theol2l4}. In this purpose, let us recall the following technical lemma.

\begin{lem}[\cite{godichon2015}]\label{lemsum}
Let $Y_{1},...,Y_{n}$ be random variables taking values in a normed vector space such that for all positive constant $q$ and for all $k \geq 1$, $\mathbb{E}\left[ \left\| Y_{k} \right\|^{q} \right] ~ < ~ \infty$. Thus, for all constants $a_{1},...,a_{n}$ and for all integer $p$,
\begin{equation}
\mathbb{E}\left[ \left\| \sum_{k=1}^{n} a_{k}Y_{k} \right\|^{p} \right] \leq \left( \sum_{k=1}^{n} \left| a_{k} \right| \left( \mathbb{E}\left[ \left\| Y_{k} \right\|^{p} \right] \right)^{\frac{1}{p}} \right)^{p}
\end{equation}
\end{lem}

\bigskip

\textbf{The remainder terms:} First, one can check that
\begin{equation}
\label{majthetgam1}\frac{1}{n^{2}}\mathbb{E}\left[ \left\| \frac{\widehat{\theta}_{1}-\theta}{\gamma_{1}} \right\|^{2} \right] = o \left( \frac{1}{n} \right) .
\end{equation}
In the same way, applying Theorem \ref{theol2l4},
\begin{align}
\notag \frac{1}{n^{2}}\mathbb{E}\left[ \left\| \frac{\widehat{\theta}_{n+1} - \theta}{\gamma_{n}} \right\|^{2} \right] & = \frac{1}{c_{\gamma}^{2}}\frac{1}{n^{2-2\alpha}}\mathbb{E}\left[ \left\| \widehat{\theta}_{n+1} - \theta \right\|^{2} \right] \\
\notag & \leq \frac{C_{1}}{c_{\gamma}^{2}}\frac{1}{n^{2-\alpha}} \\
\label{majthetgam}& = o \left( \frac{1}{n}\right) .
\end{align}
Moreover, since $ \gamma_{k}^{-1} - \gamma_{k-1}^{-1} \leq 2\alpha c_{\gamma}^{-1}k^{\alpha -1}$, applying Lemma \ref{lemsum},
\begin{align}
\notag \frac{1}{n^{2}}\mathbb{E}\left[ \left\| \sum_{k=2}^{n}\left( \frac{1}{\gamma_{k}} - \frac{1}{\gamma_{k-1} }\right) \left( \widehat{\theta}_{k} - \theta \right) \right\|^{2} \right] & \leq \frac{1}{n^{2}}\left( \sum_{k=2}^{n} \left( \frac{1}{\gamma_{k}} - \frac{1}{\gamma_{k-1}}\right) \sqrt{\mathbb{E}\left[ \left\| \widehat{\theta}_{k} - \theta \right\|^{2} \right]} \right)^{2} \\ 
\notag & \leq \frac{4\alpha^{2}c_{\gamma}^{-2}C_{1}}{n^{2}} \left( \sum_{k=2}^{n} \frac{1}{k^{1-\alpha /2}} \right)^{2} \\
\notag & = O \left( \frac{1}{n^{2 - \alpha}} \right) \\
\label{majsumgam} & = o \left( \frac{1}{n} \right) .
\end{align}
Thanks to Lemma \ref{majdelta}, there is a positive constant $C_{\theta}$ such that for all $n \geq 1$,
\[
\left\| \delta_{n} \right\| \leq C_{\theta} \left\| \widehat{\theta}_{n} - \theta \right\|^{2}.
\]
Thus, applying Lemma \ref{lemsum} and Theorem \ref{theol2l4}, there is a positive constant $C_{2}$ such that
\begin{align*}
\frac{1}{n^{2}}\mathbb{E}\left[ \left\| \sum_{k=1}^{n} \delta_{k} \right\|^{2} \right] & \leq \frac{1}{n^{2}} \left( \sum_{k=1}^{n} \sqrt{ \mathbb{E}\left[ \left\| \delta_{k} \right\|^{2}\right]} \right)^{2} \\
& \leq \frac{C_{\theta}^{2}}{n^{2}}\left( \sum_{k=1}^{n} \sqrt{\mathbb{E}\left[ \left\| \widehat{\theta}_{k} - \theta \right\|^{4} \right]} \right)^{2} \\
& \leq \frac{C_{\theta}^{2}C_{2}}{n^{2}}\left( \sum_{k=1}^{n} \frac{1}{k^{\alpha}} \right)^{2} \\
& = O \left( \frac{1}{n^{2\alpha}} \right) \\
& = o \left( \frac{1}{n} \right) . 
\end{align*}
Let $U_{n+1}:= \nabla_{y} g\left( X_{n+1} , \widehat{\theta}_{n} \right)$, note that if $\widehat{\theta}_{n} - \gamma_{n}U_{n+1} \in \mathcal{K}$, then $r_{n}=0$. Thus, applying Lemma \ref{lemsum} and Cauchy-Schwarz's inequality,
\begin{align*}
\frac{1}{n^{2}}\mathbb{E}\left[ \left\| \sum_{k=1}^{n} \frac{r_{k}}{\gamma_{k}} \right\|^{2} \right] & \leq \frac{1}{n^{2}} \left( \sum_{k=1}^{n} \frac{1}{\gamma_{k}} \sqrt{\mathbb{E}\left[ \left\| r_{k} \right\|^{2}\right]} \right)^{2} \\
& = \frac{1}{n^{2}} \left( \sum_{k=1}^{n}\frac{1}{\gamma_{k}} \sqrt{\mathbb{E}\left[ \left\| r_{k} \right\|^{2}\mathbf{1}_{\widehat{\theta}_{k} - \gamma_{k}U_{k+1} \notin \mathcal{K} }\right]} \right)^{2} \\
 & \leq \frac{1}{n^{2}}\left( \sum_{k=1}^{n}\frac{1}{\gamma_{k}}\left( \mathbb{E}\left[ \left\| r_{k} \right\|^{4}\right]\right)^{\frac{1}{4}}\left( \mathbb{P}\left[ \widehat{\theta}_{k}- \gamma_{k}U_{k+1} \notin \mathcal{K} \right] \right)^{\frac{1}{4}} \right)^{2}.
\end{align*}
Moreover, since $\pi$ is $1$-lipschitz,
\begin{align*}
\left\| r_{n} \right\|^{4} & = \left\| \pi \left( \widehat{\theta}_{n} - \gamma_{n}U_{n+1} \right) - \theta + \theta - \widehat{\theta}_{n} + \gamma_{n}U_{n+1} \right\|^{4} \\
& \leq \left( \left\| \pi \left( \widehat{\theta}_{n} - \gamma_{n}U_{n+1} \right) - \pi \left( \theta \right) \right\| + \left\| \widehat{\theta}_{n} - \gamma_{n}U_{n+1} - \theta \right\| \right)^{4} \\ 
& \leq \left( 2 \left\| \widehat{\theta}_{n} - \theta - \gamma_{n} U_{n+1} \right\| \right)^{4} \\
& \leq 2^{7} \left\| \widehat{\theta}_{n} - \theta \right\|^{4} + 2^{7}\gamma_{n}^{4}\left\| U_{n+1} \right\|^{2}. 
\end{align*}
Thus, applying inequality (\ref{majun}), there are positive constants $A_{1},A_{2}$ such that for all $n \geq 1$,
\[
\mathbb{E}\left[ \left\| r_{n} \right\|^{4}\Big|\mathcal{F}_{n} \right] \leq A_{1}\left\| \widehat{\theta}_{n} - \theta \right\|^{4} + A_{2}\gamma_{n}^{4}.
\]
In a particular case, applying Theorem \ref{theol2l4}, there is a positive constant $A_{3}$ such that for all $n \geq 1$,
\[
\mathbb{E}\left[ \left\| r_{n} \right\|^{4} \right] \leq \frac{A_{3}}{n^{2\alpha}}.
\]
Moreover, applying Theorem \ref{theol2l4}, there is a positive constant $C_{6}$ such that for all $n \geq 1$,
\[\mathbb{P}\left[ \widehat{\theta}_{n} - \gamma_{n}U_{n+1} \notin \mathcal{K} \right] \leq \frac{C_{6}}{d_{\min}^{12}n^{6\alpha}}.
\]
Then,
\begin{align}
\notag \frac{1}{n^{2}}\mathbb{E}\left[ \left\| \sum_{k=1}^{n} \frac{r_{k}}{\gamma_{k}} \right\|^{2} \right] & \leq \frac{\sqrt{C_{6}A_{3}}}{d_{\min}^{6}c_{\gamma}n^{2}}\left( \sum_{k=1}^{n}\frac{1}{k^{\alpha}} \right)^{2} \\
\notag & = O \left( \frac{1}{n^{2\alpha}} \right) \\
\label{majsumrn}& = o \left( \frac{1}{n} \right) .
\end{align}

\textbf{The martingale term:} Since $\left( \xi_{n} \right)$ is a sequence of martingale differences adapted to the filtration $\left( \mathcal{F}_{n} \right)$,
\begin{align*}
\frac{1}{n^{2}}\mathbb{E}\left[ \left\| \sum_{k=1}^{n}\xi_{k+1} \right\|^{2} \right] & = \frac{1}{n^{2}}\sum_{k=1}^{n} \mathbb{E}\left[ \left\| \xi_{k+1} \right\|^{2} \right] + \frac{2}{n^{2}}\sum_{k=1}^{n} \sum_{k'=k+1}^{n} \mathbb{E}\left[ \left\langle \xi_{k+1},\xi_{k'+1} \right\rangle \right] \\
& = \frac{1}{n^{2}}\sum_{k=1}^{n} \mathbb{E}\left[ \left\| \xi_{k+1} \right\|^{2} \right] + \frac{2}{n^{2}}\sum_{k=1}^{n} \sum_{k'=k+1}^{n} \mathbb{E}\left[ \left\langle \xi_{k+1},\mathbb{E}\left[ \xi_{k'+1}|\mathcal{F}_{k'} \right] \right\rangle \right] \\
& =  \frac{1}{n^{2}} \sum_{k=1}^{n}\mathbb{E}\left[ \left\| \xi_{k+1} \right\|^{2} \right] .
\end{align*}
Moreover,
\begin{align*}
\mathbb{E}\left[ \left\| \xi_{n+1} \right\|^{2} \right] & = \mathbb{E}\left[ \left\| U_{n+1} \right\|^{2} - 2 \mathbb{E}\left[  \left\langle \mathbb{E}\left[ U_{n+1}|\mathcal{F}_{n} \right] , \Phi (\widehat{\theta}_{n} \right\rangle \right] + \mathbb{E}\left[  \left\| \Phi (\widehat{\theta}_{n} \right) \right\|^{2} \right] \\
& = \mathbb{E}\left[ \left\| U_{n+1} \right\|^{2} \right] - \mathbb{E}\left[ \left\| \Phi (\widehat{\theta}_{n} ) \right\|^{2} \right] \\
& \leq \mathbb{E}\left[ \left\| U_{n+1}  \right\|^{2} \right] .
\end{align*}
Finally, applying inequality (\ref{majun}) and Theorem \ref{theol2l4}, there is a positive constant $M$ such that
\begin{align*}
\mathbb{E}\left[ \left\| \xi_{n+1} \right\|^{2} \right] & \leq A_{1,1}\mathbb{E}\left[ \left\| \widehat{\theta}_{n} - \theta \right\|^{2} \right] + A_{2,1} \\
& \leq M. 
\end{align*}
Then,
\begin{align*}
\frac{1}{n^{2}}\mathbb{E}\left[ \left\| \sum_{k=1}^{n} \xi_{k+1} \right\|^{2} \right] & \leq \frac{1}{n^{2}}\sum_{k=1}^{n} M \\
& = \frac{M}{n},
\end{align*}
which concludes the proof.
\end{proof}

\bigskip

\begin{proof}[Proof of Theorem \ref{theotlc}] Let us recall that the averaged algorithm can be written as follows
\begin{equation}
\sqrt{n} \Gamma_{\theta} \left( \overline{\theta}_{n} - \theta \right)   = \frac{1}{\sqrt{n}}\left(\frac{\widehat{\theta}_{1} - \theta}{\gamma_{1}} - \frac{\widehat{\theta}_{n+1} - \theta}{\gamma_{n}} +  \sum_{k=2}^{n}\left( \frac{1}{\gamma_{k}} - \frac{1}{\gamma_{k-1}}\right)\left( \widehat{\theta}_{k} - \theta \right)  - \sum_{k=1}^{n} \delta_{k} + \sum_{k=1}^{n} \frac{r_{k}}{\gamma_{k}} \right)  + \frac{1}{\sqrt{n}}\sum_{k=1}^{n}\xi_{k+1}.
\end{equation}
We now prove that the first terms on the right-hand side of previous equality converge in probability to $0$ and apply a Central Limit Theorem to the last one. 

\bigskip

\textbf{The remainder terms:} Applying inequalities (\ref{majthetgam1}) to (\ref{majsumrn}),
\begin{align*}
\frac{1}{\sqrt{n}}\frac{\widehat{\theta}_{1}-\theta}{\gamma_{1}} &  \xrightarrow[n\to \infty]{\mathbb{P}} 0 , \\
\frac{1}{\sqrt{n}}\frac{\widehat{\theta}_{n+1}-\theta}{\gamma_{n}} &  \xrightarrow[n\to \infty]{\mathbb{P}} 0 , \\
\frac{1}{\sqrt{n}}\sum_{k=2}^{n} \left( \frac{1}{\gamma_{k}} - \frac{1}{\gamma_{k-1}} \right) \left( \widehat{\theta}_{k} - \theta \right) & \xrightarrow[n\to \infty]{\mathbb{P}} 0 , \\
\frac{1}{\sqrt{n}}\sum_{k=1}^{n} \frac{r_{k}}{\gamma_{k}} & \xrightarrow[n\to \infty]{\mathbb{P}} 0 .
\end{align*} 

\bigskip

\textbf{The martingale term} Let $\widehat{\theta}_{n} = \left( Z_{n} , A_{n}\right)\in \mathbb{R}^{d} \times \mathbb{R}$, then $\xi_{n+1}$ can be written as $\xi_{n+1} = \xi_{n+1}' + \epsilon_{n+1} + \epsilon_{n+1}'$, with
\begin{align*}
 \xi_{n+1}' & := \begin{pmatrix}
\mu - X_{n+1} - \mathbb{E}\left[ \mu - X_{n+1} |\mathcal{F}_{n} \right] -r^{\star}\left( \frac{\mu - X_{n+1}}{\left\| \mu - X_{n+1} \right\|} - \mathbb{E}\left[ \frac{\mu - X_{n+1}}{\left\| \mu - X_{n+1} \right\|} \big| \mathcal{F}_{n} \right] \right)   \\
r^{\star} - \left\| \mu - X_{n+1} \right\| - r^{\star} + \mathbb{E}\left[ \left\| \mu - X_{n+1} \right\| \big| \mathcal{F}_{n} \right]
\end{pmatrix} , \\
\\
 \epsilon_{n+1}& : = -\begin{pmatrix}
\left( A_{n}-r^{\star}\right)\left( \frac{Z_{n} - X_{n+1}}{\left\| Z_{n} - X_{n+1} \right\|} - \mathbb{E}\left[ \frac{Z_{n} - X_{n+1}}{\left\| Z_{n} - X_{n+1} \right\|} \big|\mathcal{F}_{n} \right] \right)     \\
\left\| Z_{n} - X_{n+1} \right\| - \mathbb{E}\left[ \left\| Z_{n} - X_{n+1} \right\| \mathcal{F}_{n} \right] - \left\| \mu - X_{n+1} \right\| + \mathbb{E}\left[ \left\| \mu - X_{n+1} \right\| \big| \mathcal{F}_{n} \right] 
\end{pmatrix} , \\
\\
\epsilon_{n+1}' & := -\begin{pmatrix}
r^{\star}\left( \frac{Z_{n} - X_{n+1}}{\left\| X_{n+1} - Z_{n} \right\|} - \frac{\mu - X_{n+1}}{\left\| \mu - X_{n+1} \right\|} - \mathbb{E}\left[ \frac{Z_{n} - X_{n+1}}{\left\| X_{n+1} - Z_{n} \right\|} - \frac{\mu - X_{n+1}}{\left\| \mu - X_{n+1} \right\|} \Big| \mathcal{F}_{n} \right] \right) \\
0 \end{pmatrix} .
\end{align*}
Note that $\left( \xi_{n} \right) , \left( \epsilon_ {n} \right) , \left( \epsilon_{n}' \right)$ are martingale differences sequences adapted to the filtration ~$\left( \mathcal{F}_{n} \right) $. Thus,
\begin{align*}
\frac{1}{n}\mathbb{E}\left[ \left\| \sum_{k=1}^{n} \epsilon_{k+1} \right\|^{2} \right] & = \frac{1}{n}\sum_{k=1}^{n} \mathbb{E}\left[ \left\| \epsilon_{k+1} \right\|^{2} \right] + \frac{2}{n}\sum_{k=1}^{n} \sum_{k'=k+1}^{n} \mathbb{E}\left[ \left\langle \epsilon_{k+1} , \epsilon_{k'+1} \right\rangle \right] \\
& = \frac{1}{n}\sum_{k=1}^{n} \mathbb{E}\left[ \left\| \epsilon_{k+1} \right\|^{2} \right] + \frac{2}{n}\sum_{k=1}^{n} \sum_{k'=k+1}^{n} \mathbb{E}\left[ \left\langle \epsilon_{k+1} , \mathbb{E}\left[ \epsilon_{k'+1}|\mathcal{F}_{k'} \right] \right\rangle \right] \\
& = \frac{1}{n}\sum_{k=1}^{n} \mathbb{E}\left[ \left\| \epsilon_{k+1} \right\|^{2} \right] .
\end{align*}
Moreover,
\begin{align*}
\mathbb{E}\left[ \left\| \epsilon_{n+1} \right\|^{2} |\mathcal{F}_{n} \right] & = \mathbb{E}\left[ \left\| \begin{pmatrix}
\left( A_{n} - r^{\star} \right)  \frac{Z_{n} - X_{n+1}}{\left\| Z_{n} - X_{n+1}\right\|} \\
\left\| Z_{n} - X_{n+1} \right\| - \left\| \mu -X_{n+1} \right\|
\end{pmatrix} \right\|^{2} \Big| \mathcal{F}_{n} \right] \\
\\
&  + \left\| \begin{pmatrix}
\left( A_{n} - r^{\star} \right)  \mathbb{E}\left[ \frac{Z_{n} - X_{n+1}}{\left\| Z_{n} - X_{n+1}\right\|} \Big|\mathcal{F}_{n}  \right] \\
\mathbb{E}\left[ \left\| Z_{n} - X_{n+1} \right\| - \left\| \mu -X_{n+1} \right\| |\mathcal{F}_{n} \right]
\end{pmatrix} \right\|^{2} \\
\\
& -2 \mathbb{E}\left[ \left\langle \begin{pmatrix}
\left( A_{n} - r^{\star} \right)  \frac{Z_{n} - X_{n+1}}{\left\| Z_{n} - X_{n+1}\right\|} \\
\left\| Z_{n} - X_{n+1} \right\| - \left\| \mu -X_{n+1} \right\|
\end{pmatrix} , \begin{pmatrix}
\left( A_{n} - r^{\star} \right)  \mathbb{E}\left[ \frac{Z_{n} - X_{n+1}}{\left\| Z_{n} - X_{n+1}\right\|} \Big|\mathcal{F}_{n}  \right] \\
\mathbb{E}\left[ \left\| Z_{n} - X_{n+1} \right\| - \left\| \mu -X_{n+1} \right\| |\mathcal{F}_{n} \right]
\end{pmatrix} \right\rangle \Big| \mathcal{F}_{n} \right] .
\end{align*}
Thus, one can check that
\begin{align*}
\mathbb{E}\left[ \left\| \epsilon_{n+1} \right\|^{2}|\mathcal{F}_{n}  \right] & \leq \mathbb{E} \left[ \left\| \begin{pmatrix}
-\left( A_{n} - r^{\star} \right)  \frac{Z_{n} - X_{n+1}}{\left\| Z_{n} - X_{n+1}\right\|} \\
\left\| Z_{n} - X_{n+1} \right\| - \left\| \mu -X_{n+1} \right\|
\end{pmatrix} \right\|^{2} \Big| \mathcal{F}_{n} \right] \\
& \leq \left\| A_{n} - r^{\star} \right\|^{2} + \left\| Z_{n} - \mu \right\|^{2} \\
& = \left\| \widehat{\theta}_{n} - \theta \right\|^{2} .
\end{align*}
Thus, applying Theorem \ref{theol2l4},
\begin{align*}
\frac{1}{n}\mathbb{E}\left[ \left\| \sum_{k=1}^{n} \epsilon_{k+1} \right\|^{2} \right] & = \frac{1}{n}\sum_{k=1}^{n} \mathbb{E}\left[ \left\| \epsilon_{k+1} \right\|^{2} \right] \\
& \leq \frac{1}{n}\sum_{k=1}^{n}\mathbb{E}\left[ \left\| \widehat{\theta}_{n} - \theta \right\|^{2} \right] \\
& \leq \frac{1}{n}\sum_{k=1}^{n}\frac{C'}{n^{\alpha}} \\
& = O \left( \frac{1}{n^{\alpha}}\right) .
\end{align*}
As a particular case, 
\[
\frac{1}{\sqrt{n}}\sum_{k=1}^{n} \epsilon_{k+1} \xrightarrow[n \to \infty]{\mathbb{P}} 0.
\]
Similarly, since $\left\| \frac{Z_{n} - X_{n+1}}{\left\| Z_{n} - X_{n+1} \right\|}- \frac{\mu -X_{n+1}}{\left\| \mu - X_{n+1} \right\|}\right\| \leq 2$,
\begin{align*}
\mathbb{E}\left[ \left\| \epsilon_{n+1}' \right\|^{2}|\mathcal{F}_{n} \right] & \leq \mathbb{E}\left[ \left\| \begin{pmatrix}
r^{\star}\left( \frac{Z_{n} - X_{n+1}}{\left\| X_{n+1} - Z_{n} \right\|} - \frac{\mu - X_{n+1}}{\left\| \mu - X_{n+1} \right\|} \right) \\
0
\end{pmatrix} \right\|^{2} \Big| \mathcal{F}_{n} \right] \\
& \leq 2 \left( r^{\star} \right)^{2} \mathbb{E}\left[ \left\|  \frac{Z_{n} - X_{n+1}}{\left\| X_{n+1} - Z_{n} \right\|} - \frac{\mu - X_{n+1}}{\left\| \mu - X_{n+1} \right\|} \right\|  \Big| \mathcal{F}_{n} \right] .
\end{align*}
This last term is closely related to the gradient of the function we need to minimize to get the geometric median (see \cite{kemperman1987median} for example) and it is proved in \cite{CCG2015} that since Lemma \ref{lemsmallball} is verified, then 
\[
\mathbb{E}\left[ \left\|  \frac{Z_{n} - X_{n+1}}{\left\| X_{n+1} - Z_{n} \right\|} - \frac{\mu - X_{n+1}}{\left\| \mu - X_{n+1} \right\|} \right\|  \Big| \mathcal{F}_{n} \right] \leq C \left\| Z_{n} - \mu \right\| \leq C \left\| \widehat{\theta}_{n} - \theta \right\|.
\]
Thus,
\[
\mathbb{E}\left[ \left\| \epsilon_{n+1} \right\|^{2} |\mathcal{F}_{n} \right] \leq 2C\left( r^{\star} \right)^{2} \left\| \widehat{\theta}_{n} - \theta \right\| .
\]
Finally, since $\left( \epsilon_{n} \right)$ is a sequence of martingale differences adapted to the filtration $\left( \mathcal{F}_{n} \right)$, applying Theorem \ref{theol2l4} and Cauchy-Schwarz's inequality,
\begin{align*}
\frac{1}{n}\mathbb{E}\left[ \left\| \sum_{k=1}^{n} \epsilon_{k+1}' \right\|^{2} \right] & = \frac{1}{n}\sum_{k=1}^{n} \mathbb{E}\left[ \left\| \epsilon_{k+1}' \right\|^{2} \right] \\
& \leq 2C\left( r^{\star} \right)^{2} \frac{1}{n} \sum_{k=1}^{n}\mathbb{E} \left[ \left\| \widehat{\theta}_{n} - \theta \right\| \right] \\
& \leq 2C\left( r^{\star} \right)^{2}\sqrt{C_{1}} \frac{1}{n} \sum_{k=1}^{n} \frac{1}{k^{\alpha /2}} \\
& = O \left( \frac{1}{n^{\alpha /2}} \right) .
\end{align*}
Note that with more assumptions on $W$, we could get a better rate but this one is sufficient. Indeed, thanks to previous inequality,
\[
\frac{1}{\sqrt{n}}\sum_{k=1}^{n} \epsilon_{k+1}' \xrightarrow[n \to \infty]{\mathbb{P}} 0 .
\]
Finally, applying a Central Limit Theorem (see \cite{DUF1997} for example), we have the convergence in law
\begin{equation}
\frac{1}{\sqrt{n}}\sum_{k=1}^{n} \xi_{k+1}' \xrightarrow[n \to \infty]{\mathcal{L}} \mathcal{N}\left( 0 , \Sigma \right),
\end{equation}
with 
\[
\Sigma := \mathbb{E}\left[ \begin{pmatrix}
\mu -X - r^{\star}\frac{\mu -X}{\left\| \mu - X \right\|} \\
r^{\star} - \left\| \mu - X \right\|
\end{pmatrix}\otimes \begin{pmatrix}
\mu -X - r^{\star}\frac{\mu -X}{\left\| \mu - X \right\|} \\
r^{\star} - \left\| \mu - X \right\| 
\end{pmatrix} \right] ,
\]
which also can be written as
\[
\Sigma = \mathbb{E}\left[ \begin{pmatrix}
r^{\star}U_{\Omega} - rWU_{\Omega}   \\
r^{\star} - rW
\end{pmatrix}\otimes \begin{pmatrix}
r^{\star}U_{\Omega} - rWU_{\Omega} \\
r^{\star} - rW
\end{pmatrix} \right] .
\]
Thus, we have the convergence in law
\[
\sqrt{n}\Gamma_{\theta} \left( \overline{\theta}_{n} - \theta \right) \xrightarrow[n \to \infty]{\mathcal{L}} \mathcal{N}\left( 0 , \Sigma \right),
\]
and in a particular case,
\[
\sqrt{n} \left( \overline{\theta}_{n} - \theta \right) \xrightarrow[n \to \infty]{\mathcal{L}} \mathcal{N}\left( 0 , \Gamma_{\theta}^{-1} \Sigma \Gamma_{\theta}^{-1} \right) . 
\]
\end{proof}

\def\cprime{$'$}

\end{document}